\documentclass[12pt]{amsart}
\usepackage{latexsym, amsmath, amssymb, amsthm, mathrsfs,wasysym,mathtools}
\usepackage[alwaysadjust]{paralist}
\usepackage{caption}
\usepackage{subcaption}
\usepackage[T1]{fontenc}
\usepackage{lmodern}
\usepackage[backref,colorlinks=true,linkcolor=blue,urlcolor=blue, citecolor=blue]%
  {hyperref}
\usepackage[shortalphabetic]{amsrefs}
\usepackage[all]{xy}
\usepackage[T1]{fontenc}
\usepackage{mathptmx}
\usepackage{extarrows}
\usepackage{microtype}
\usepackage{framed}
\usepackage{tikz}
\usepackage{tikz-cd}
\usepackage{graphicx}
\usepackage[alwaysadjust]{paralist}
\makeatletter
\providecommand\@enum@widestlabel{7}
\makeatother

\usepackage[centering, includeheadfoot, hmargin=.75in, vmargin=.75in,
  headheight=30.4pt]{geometry}
\newtheorem{lemma}{Lemma}[section]
\newtheorem{theorem}[lemma]{Theorem}

\newtheorem{proposition}[lemma]{Proposition}

\theoremstyle{definition}

\newtheorem{remark}[lemma]{Remark}
\newtheorem*{remark*}{Remark}

\renewcommand{\theequation}%
{\arabic{section}.\arabic{lemma}.\arabic{equation}}
\newcommand{\PP}{\ensuremath{\mathbb{P}}}

\newcommand{\sI}{\ensuremath{\kern -1pt \mathscr{I}\kern -2pt}} 
\newcommand{\sJ}{\ensuremath{\kern -2pt \mathscr{J}\kern -2pt}}

\newcommand{\bb}{\ensuremath{\mathfrak{b}}}

\renewcommand{\geq}{\geqslant}
\renewcommand{\leq}{\leqslant}

\DeclareMathOperator{\mult}{mult}
\DeclareMathOperator{\Nef}{Nef}
\DeclareMathOperator{\Eff}{\overline{Eff}}
\DeclareMathOperator{\Mov}{Mov}

\DeclareMathOperator{\vol}{vol}

\newcommand{\deq}{\ensuremath{\stackrel{\textrm{def}}{=}}}

\newcommand{\nob}[2]{\ensuremath{\Delta_{#1}(#2)}}
\newcommand{\inob}[2]{\ensuremath{{\Delta}_{#1}(#2)}}

\input xy
\xyoption{all}

\usepackage{comment}
\usepackage{framed}
\usepackage{color}
\definecolor{shadecolor}{gray}{0.875}

\let\cal\mathcal
\let\frak\mathfrak
\let\bb\mathbb

\begin{document}

\title{Fujita--Zariski decompositions on some product threefolds}

\author{Mihai Fulger}
\address{University of Connecticut, Department of Mathematics, Storrs CT 06269, USA}
\address{Institute of Mathematics of the Romanian Academy, Bucharest 010702, Romania}
\email{mihai.fulger@uconn.edu}

\author{Victor Lozovanu}
\address{Dipartimento di Matematica, Universit\`a Degli Studi Di Genova, Genova 16146, Italy}
\email{victor.lozovanu@unige.it}

\begin{abstract}
	We use explicit blow-ups and computations of birational Fujita--Zariski decompositions to determine generic infinitesimal Newton--Okounkov bodies for box-product ample polarizations on three classes of spaces: product between a curve and the projective plane, products of three curves, and the product between a curve and a Jacobian surface. In particular we compute the volume of many big but non-nef divisors on blow-ups of these threefolds.
\end{abstract}

\maketitle

\section{Introduction}

\subsection{The Riemann--Roch problem}
Given a smooth projective variety $X$ of dimension $n$ over an algebraically closed field, and a line bundle $L$ on it, the Riemann--Roch problem asks to compute the Hilbert function $m\mapsto \dim H^0(X,mL)$. This is often difficult, and instead we settle on measuring the asymptotic growth of the function, more precisely on computing the \emph{volume}
\[{\rm vol}(L)\deq\limsup_{m\to\infty}\frac{\dim H^0(X,mL)}{m^n/n!}\ .\]
This is a numerical birational invariant of $L$, well-studied in \cite{Laz04}. When $L$ satisfies strong positivity conditions like being ample, or at least \emph{nef} (i.e., asymptotically ample), then $\vol(L)=(L^n)$ is the top self-intersection. The volume is positive precisely when $L$ is \emph{big} (i.e., birationally ample), however there are plenty of line bundles that are big and not nef and then the computation of their volume is unclear.

For surfaces, the problem above is solved by Zariski decompositions. We can decompose $L=P+N$ as divisors, where $P$ is nef and $N$ effective and $\vol(L)=\vol(P)=(P^2)$. Computing the Zariski decomposition is effective when an effective divisor that represents $L$ is known.

In higher dimension, Zariski decompositions can be tricky. In general we decompose $L=P_{\sigma}(L)+N_{\sigma}(L)$, the so called \emph{Nakayama $\sigma$-Zariski decomposition} of $L$ (see \cite{Nak04}). We still have $\vol(L)=\vol(P_{\sigma}(L))$, however now $P_{\sigma}(L)$ is only \emph{movable} (i.e., asymptotically it deforms without fixed divisorial components), not necessarily nef, and so the intersection theoretic interpretation for the volume is missing.

Sometimes (but not always), for example on toric varieties, we can find an explicit birational model $\rho:\widetilde X\to X$, after possibly many blow-ups, such that the positive part of $\rho^*L$ is nef. We call this a \emph{birational Fujita--Zariski} decomposition. It is also an example of a Cutkosky--Kawamata--Moriwaki decomposition (see \cite{Pro03,KM13}). When this happens, then $\vol(L)=(P_{\sigma}(\rho^*L)^n)$.

\subsection{Generic infinitesimal Newton--Okounkov bodies}
An even more ambitious problem than finding the volume is the computation of various Newton--Okounkov bodies (\emph{NObodies}) associated to big line bundles $L$. We focus here on \emph{infinitesimal} NObodies.  If $x\in X$ and $\pi:{\rm Bl}_xX\to X$ is the blow-up with exceptional divisor $E\deq \pi^{-1}\{x\}\simeq\bb P^{n-1}$, then from a complete linear flag $Y_{\bullet}:\ \bb P^{n-1}\simeq E=Y_1\supset Y_2\supset\ldots\supset Y_n=\{y\}$ one can associate $\nob{Y_{\bullet}}{\pi^*L}\subset\bb R^n_+$, the \emph{infinitesimal Newton--Okounkov body (iNObody)} of $L$ at $x$ associated to $Y_{\bullet}$. It is a valuative asymptotic construction from sections of all $mL$. We detail it in \S\ref{ss:NObodies}. 
The simplest shapes are obtained when $Y_{\bullet}$ is a very general linear flag, and the resulting body is independent of this very general choice (cf. \cite{LM09}). We call it the \emph{generic iNObody} and denote it $\nob{x}{L}$.

A feature of the valuative construction is that ${\rm vol}_{\bb R^n}\nob{x}{L}=\frac 1{n!}\vol(L)$, thus computing $\nob{x}{L}$ is indeed a more general problem. In fact, if we consider coordinates $(\nu_1,\ldots,\nu_n)$ on $\bb R^n$, then $\vol_{\bb R^n}\inob{x}{L}_{\nu_1\geq t}=\frac 1{n!}\vol(\pi^*L-tE)$. Thus $\inob{x}{L}$ captures the volume of the entire ray $t\geq 0$ of divisors
\[L_t\deq \pi^*L-tE\]
on ${\rm Bl}_xX$. Computing the body is already an interesting problem for ample line bundles $L$ since the $L_t$ are not necessarily nef.

On surfaces, we explain in Proposition \ref{prop:NOsurface} via work of \cite{LM09} that Zariski decompositions are again key for computing generic iNObodies. If $L$ is big and $x\not\in{\bf B}_-(L)$, e.g., when $L$ is also nef, then 
\[\inob{x}{L}\ =\ \text{region below the graph of }t\mapsto P_{\sigma}(L_t)\cdot E\ .\]
Thus, in this case, the Zariski decompositions for all divisors $L_t$ on ${\rm Bl}_xX$ determine the generic iNObody.

Before work of the authors in \cite{FLgeneralities,FLproductcurves,FLjac}, essentially the only example in the literature of a generic iNObody computation on a variety of dimension $n\geq 3$ was on $\bb P^n$, where the result is simplicial. A difficulty is the genericity condition on the linear flag $Y_{\bullet}$ which makes the computation unclear even when $X$ is a smooth toric variety and $x$ a torus-fixed point. We explain in \cite{FLgeneralities} that if $X=\bb P^1\times\bb P^1$, if $L=\cal O(1,1)$, and $x$ is a torus-fixed point (in fact any point), then the generic $\inob{x}{L}$ differs from the body computed with respect to the linear flags $E=Y_1\supset Y_2$ for both choices of points $Y_2\in E$ that are fixed by the torus in ${\rm Bl}_xX$.

In this paper, as in \cite{FLjac}, our idea is that the strategy on surfaces should extend to threefolds when we can find a sequence of blow-ups $\rho:\widetilde X\to {\rm Bl}_xX$ such that $P_{\sigma}(\rho^*L_t)$ is nef for all $t$, i.e., a simultaneous birational Fujita--Zariski decompositions for the entire ray $L_t=\pi^*L-tE$ of divisors on ${\rm Bl}_xX$. When this happens, we prove that the vertical slice $\nob{x}{L}_{\nu_1=t}$ is a NObody for the restriction of $P_{\sigma}(\rho^*L_t)$ to the strict transform $\widetilde E$ of $E$ in $\widetilde X$.  The surface $\widetilde E$ is a sequence of point blow-ups of $\bb P^2$, and computations are very concrete when the nef restricted divisor $P_{\sigma}(\rho^*L_t)|_{\widetilde E}$ is explicitly presented as an effective divisor. We are able to see this idea through in three classes of examples, adding to the Jacobian threefold examples in the upcoming \cite{FLjac}. 

\begin{theorem}
	Assume we are in one of the following cases:
	\begin{enumerate}
		\item $X=C\times\bb P^2$ where $C$ is a smooth projective curve, and $L=L_1\boxtimes\cal O_{\bb P^2}(b)$ is a box-product of ample line bundles.
		\smallskip
		
		\item $X=C_1\times C_2\times C_3$ where $C_i$ are smooth projective curves and $L$ is a box-product $L=L_1\boxtimes L_2\boxtimes L_3$ where $L_i$ are line bundles on $C_i$.
		\smallskip
		
		\item $X=C\times {\rm Jac}(D)$ where $C$ is any smooth projective curve, while $D$ is a smooth projective curve of genus 2, and $L$ is again a box-product $L_1\boxtimes \cal O(a\theta)$, where $\theta$ is the Theta divisor on the Jacobian.
	\end{enumerate}
Then, for every $x\in X$, explicit birational Fujita--Zariski decompositions exist for the entire ray of divisors $L_t=\pi^*L-E$ on ${\rm Bl}_xX$ and in particular we compute the function $t\mapsto{\rm vol}(L_t)$. Furthermore we describe the convex body $\inob{x}{L}$. 
\end{theorem}

Theorems \ref{thm:CP}, \ref{thm:CCCfinalform} and \ref{thm:JxCfinalform} have our precise results and a few illustrations of resulting generic iNObodies. We hope that the computations will serve as testing ground for various open problems. For instance, in Theorem \ref{thm:Seshdaricurveconjecture} we see that our examples give counterexamples to a question raised by \cite{FLgeneralities}.

The first two examples have a very toric feel. The genera of $C_i$ do not affect our computations. In the first case, to find birational Fujita--Zariski decompositions we further blow-up ${\rm Bl}_xX$ along a copy of $C$, the strict transform of the fiber of the second projection through $x$. In the second case, we further blow-up ${\rm Bl}_xX$ along the strict transforms of the three ``coordinate curves'' $C_i$ through $x$. The surface $\widetilde E$ on which we compute the vertical slices $\nob{x}{L}_{\nu_1=t}$ is the blow-up of $\bb P^2$ in one point in the first case, and in $3$ non-collinear points in the second case. We get help from the geometry of the Cremona transform, and from the toric structure.

The third example is more complicated. We compute Nakayama $\sigma$-Zariski decompositions on ${\rm Bl}_xX$ by first computing the movable cone of divisors on this space. As in \cite{FLjac}, the blow-ups are then guided by the idea that we want to find explicit effective representatives for the particular $L_{\nu}=\pi^*L-\nu\cdot E$ on the boundary of the movable cone, and then blow-up until their strict transforms become nef. This process then ends up computing birational Fujita--Zariski decompositions not just for $L_{\nu}$, but for all $L_t$.
The blow-ups are more involved and numerous. There are two important curves through $x$ on $X$: a copy of $C$ appearing as a fiber of the projection to ${\rm Jac}(D)$, and a curve $R\times \{c_0\}$ where $R\subset{\rm Jac}(C)$ is the Seshadri curve, the image of an Abel--Jacobi embedding $C\subset{\rm Jac}(C)$ under the multiplication by two map on the Jacobian. We blow-up the strict transform of $C$ once and then the strict transform of $R\times\{c_0\}$ three times in a natural sense. In general, blow-ups seem to terminate when the curves that we blow-up all have semistable normal bundles. The rational surface $\widetilde E$ on which we compute slices of the iNObody is the surface in Lemma \ref{lem:P2blow7}, which is $\bb P^2$ blown-up in 7 points, 6 of which are on a line, then two more blow-ups, each time of 6 infinitesimally near points to the last 6.

\subsection{Other computations}
One can ask whether our strategy extends in higher dimension. 
The answer is yes when it comes simply to computing Fujita--Zariski decompositions. For example they should be easy to do on any concrete toric example. However, when it comes to using them to compute generic iNObodies, already in dimension 4 a new level of complexity is introduced. For the computation of the slices $\inob{x}{L}_{\nu_1=t}$, we need some NObodies for $(P_{\sigma}(\rho^*L_t)+s\rho^*E)|_{\widetilde E}$
on the threefold $\widetilde E$. A new parameter $s$ appears and new blow-ups of $\widetilde E$ may be needed (and it is not clear that only finitely many exist) in order to run this computation. Clearly higher dimensions add to the complexity of the problem.

Setting aside birational Fujita--Zariski decompositions, sometimes effective computations of generic iNObodies can be carried out in higher dimension via a guess-and-check method. We do this in \cite{FLproductcurves} for balanced box-products $\boxtimes_{i=1}^nL_i$ on products of arbitrarily many curves, but we ask for example $\deg L_i=1$ for all $i$. The result is a simplex. The computation again benefits from toric geometry.

In \cite{FLgeneralities} we initiate a deep investigation into the shape of generic iNObodies. We prove that they satisfy a certain property that we call Borel-fixed, at least when $x$ is very general in $X$. It implies sharp upper and lower-bounds. We extract a characterization for when the generic iNObody is simplicial in \cite{FLgeneralities}. As application we recovered the case of balanced box-products on products of curves from \cite{FLproductcurves}, the case of non-hyperelliptic Jacobian threefolds from \cite{FLjac}, but also other examples such as quadrics in projective space, or symmetric products of curves with polarization descended from a balanced box-product.

\subsection*{Acknowledgments} The first named author was partially supported by the Simons travel grant no.~579353.  The second named author was partially supported by the Research Project PRIN 2020 - CuRVI, CUP J37G21000000001, and PRIN 2022 - PEKBY, CUP J53D23003840006. Furthermore, he wishes to thank the MIUR Excellence Department of Mathematics, University of Genoa, CUP D33C23001110001. He is also a member of the INDAM-GNSAGA. 

The polytope pictures were created in \cite{Mathematica}. Figure \ref{fig:P27} was created in GoodNotes, while for Figure \ref{fig:CxJ} we also used Notability. We used the \cite{Polymake} software (see also \cite{GJ97}) to convert between polytopes given as convex hulls and polytopes described by inequalities.

\section{Notation, comments, preliminaries}

\subsection{Positive part of a real number}\label{subsec:r+}
For a real number $r$, denote
\[r_+\deq\max\{0,r\}.\]

\subsection{Positivity notions for divisors}
We work with classes in $N^1(X)_{\bb R}$. Our positive classes usually form convex cones and we transfer said positivity property to and from the cone.

The closure of the ample cone ${\rm Amp}(X)$ is the \emph{nef} cone $\Nef(X)$. The closure of the cone of effective divisors is the \emph{pseudoeffective} cone $\Eff(X)$. Its interior is the \emph{big} cone. The closure of the cone generated by Cartier divisors with complete linear series without fixed divisorial components is the \emph{movable} cone $\Mov(X)$. We have inclusions $\Nef(X)\subset\Mov(X)\subset\Eff(X)$.

Nefness is preserved by pullback and by finite flat pushforward, movability is preserved by generically finite dominant pushforward, while pseudoeffectivity is preserved by dominant pullback and generically finite dominant pushforward.

\subsection{Positivity on point blow-ups}
Let $\pi:\overline{X}\rightarrow X$ be the blow-up of a smooth projective variety of dimension $n$ at a point $x$ with $E\simeq \PP^{n-1}$ the exceptional divisors. Let $L$ be an ample line bundle on $X$. Put 
\[
L_t \ \deq \ \pi^*L-tE, \textup{ for any } t\geq 0 \ .
\]
We have the following local positivity invariants: 
\begin{itemize}
	\item $\epsilon\deq \epsilon(L;x)=\max\{t\geq 0\ \mid\ L_t\text{ is nef}\}$ is the \textit{Seshadri constant}; 
	\item $\nu\deq\nu(L;x)=\max\{t\geq 0\ \mid\ L_t\text{ is movable}\}$; 
	\item $\mu\deq \mu(L;x)=\max\{t\geq 0\ \mid\ L_t\text{ is psuedoeffective}\}$ is the \emph{infinitesimal width}, or \emph{Fujita--Nakayama} invariant. 
\end{itemize} If $L$ is ample, then the class $L_t$ is ample for $t\in (0,\epsilon)$, movable but not nef in $(\epsilon,\nu]$, and big but not movable for $t\in (\nu, \mu)$.

\subsection{Techniques for checking positivity}

\subsubsection{A simple test for nefness}\label{subsec:neftest} Let $D=P+\sum_{i=1}^ma_iY_i$ with $P$ nef, with $a_i\geq 0$, and $Y_i$ prime divisors. Let $\pi_i:Y'_i\to X$ be resolutions of the $Y_i$, or we can simply ask that the $Y'_i$ are smooth projective and dominate $Y_i$. Then $D$ is nef iff $\pi_i^*D$ is nef for all $i$.

\subsubsection{A movability criterion}\label{subsec:movtest}
If $D=P+N=P'+N'$ with $P,P'$ movable and $N,N'$ effective such that the supports of $N,N'$ do not share divisorial components, then $D$ is movable.

\subsubsection{A necessary condition for movability}\label{subsec:movcond}
If $\pi:Y'\to X$ is a smooth alteration of a prime divisor on $X$ and $D$ is movable, then $\pi^*D$ is pseudoeffective. 

\subsection{Zariski decompositions}\label{subsec:Nakayama}

If $D$ is a big $\mathbb R$-Cartier $\mathbb R$-divisor on $X$, then there exists a unique decomposition
\[D=P_{\sigma}(D)+N_{\sigma}(D)\]
such that $P_{\sigma}(D)$ is movable, $N_{\sigma}(D)$ is effective, and $D'-N_{\sigma}(D)$ is an effective divisor for every effective $\bb R$-Cartier $\bb R$-divisor $D'\equiv D$. This is the \emph{Nakayama $\sigma$-Zariski decomposition} of $D$. When $X$ is a surface, this is just the usual Zariski decomposition.
Write 
\[N_{\sigma}(D)=\sum\nolimits_Y\sigma_Y(D)\cdotp Y,\] where the sum ranges through all prime divisors $Y$ in $X$. For given divisor $D$, only at most finitely many $\sigma_Y(D)$ are nonzero, thus the sum is well-defined.
We list some properties that we will use:
\begin{enumerate}
	\item If $D=P+N$ such that $P$ is movable and $N$ is effective, then $N-N_{\sigma}(D)$ is an effective divisor. 
	\smallskip
	
	\item If $\pi:\overline Y\to X$ is a resolution of a prime divisor $Y$ on $X$ and $u=\inf\{t\geq0\ \mid \pi^*(D-tY)\text{ is pseudoeffective}\}$, then $\sigma_Y(D)\geq u_+$. 
	\smallskip
	
	\item If $\pi:Y\to X$ is birational, then $P_{\sigma}(\pi^*D)=P_{\sigma}(\pi^*P_{\sigma}(D))$ and their cycle pushforward in $X$ is $P_{\sigma}(D)$. The divisor  $\pi^*P_{\sigma}(D)-P_{\sigma}(\pi^*D)=N_{\sigma}(\pi^*P_{\sigma}(D))$ 
	is effective and $\pi$-exceptional. It can be nonzero.
	\smallskip
	
	\item ${\rm Supp}\, N_{\sigma}(D)$ is the divisorial part of ${\bf B}_-(D)$. See \cite{ELMNP06} for the definition and properties of ${\bf B}_{\pm}(D)$.
\end{enumerate}

\subsubsection{Birational Fujita--Zariski decomposition}
In dimension $n\geq 3$, the positive part $P_{\sigma}(D)$ is in general not nef. If there exists $\rho:\widetilde X\to X$ birational with $\widetilde X$ smooth such that $P_{\sigma}(\rho^*D)$ is nef, we say that the big divisor $D$ admits a \emph{birational Fujita--Zariski decomposition}. This is also a particular case of a birational \emph{Cutkosky--Kawamata--Moriwaki decomposition}. See \cite{Pro03,KM13} for extensions of the theory to pseudoeffective divisors.

\subsection{Newton--Okounkov bodies}\label{ss:NObodies}

Fix a flag 
\[X=Y_0\supset Y_1\supset\ldots\supset Y_n=\{x\}\] 
such that each $Y_i$ is smooth irreducible of dimension $n-i$, or at least smooth at $x$. If $D$ is any nonzero  divisor on $X$, put
\[\nu_1(D)\ \deq\ \textup{ord}_{Y_1}(D).\]
Then $D-\nu_1(D)Y_1$ is a divisor that does not contain $Y_1$ in its support. It has a well-defined restriction $D-\nu_1(D)Y_1|_{Y_1}$ on $Y_1$. Put
\[\nu_2(D)\ \deq\ \textup{ord}_{Y_2}(D-\nu_1(D)Y_1|_{Y_1}).\]
Continue inductively to define
\[\nu\ \deq\ (\nu_1,\ldots,\nu_n).\]
Note that if $D$ is effective, then $\nu(D)\in\bb Z^n_{\geq 0}$. The function $\nu$ is valuative. In particular
$\nu(D+D')=\nu(D)+\nu(D')$
and if $V\subset H^0(X,D)$ is any linear series that we see as a set of effective Cartier divisors on $X$, then
$\#\nu(V\setminus\{0\})=\dim V.$
For $D$ a big line bundle, the \emph{Newton--Okounkov} body of $D$ is 
\[\nob{Y_{\bullet}}{D}\ \deq\ \overline{\bigcup\nolimits_{m\geq 1}\frac 1m\nu(H^0(X,mD)\setminus\{0\})}.\]
Fixing a flag $Y_{\bullet}$ as above, \cite{LM09} prove that $\nob{Y_{\bullet}}{D}$ only depends on the numerical class of $D$, and $n!\cdot {\rm vol}_{\bb R^n}\nob{Y_{\bullet}}{D}=\vol(D)$.

The \emph{generic infinitesimal Newton--Okounkov body} is
\[\inob{x}{D}\ \deq\ \nob{\overline Y_{\bullet}}{\pi^*D},\]
where $\overline Y_{\bullet}$ is a very general linear flag in $E\simeq\bb P^{n-1}$ with $\overline Y_i\simeq\bb P^{n-i}$ for $i>0$ and $\overline Y_0=\overline X$. This is well-defined by \cite{LM09}.
Specific interest in the generic infinitesimal NObody comes from the fact that it can be used to test whether $x\in{\bf B}_{\pm}(D)$, and it recovers the \emph{moving Seshadri constant} $\epsilon(||D||;x)$ when $x\not\in{\bf B}_+(D)$. These were proved in \cite{KL17}. In \cite{FLgeneralities} we prove that additionally the generic iNObody has a special shape that we call Borel-fixed, and in particular we recover sharp bounds that are simplicial from below and of a special polytopal shape from above.

\begin{proposition}[The surface case, cf. {\cite[Theorem 6.4]{LM09}} ]\label{prop:NOsurface}
	Let $X$ be a smooth projective surface and $D$ a big divisor on $X$.
	Let $Y_{\bullet}$ be the flag $X\supset C\supset\{x\}$ where $C$ is a smooth curve. Put $\mu=\sup\{t\geq 0\ \mid\ D-tC\text{ is big}\}$. Assume that $C$ is not contained in the support of $N_{\sigma}(D)$. Then
	\[\Delta_{Y_{\bullet}}(D)=\overline{\bigl\{(t,x)\in\bb R^2_{\geq 0}\ \mid\ 0\leq t<\mu,\text{ and }x\in[\alpha(t),\beta(t)]\bigr\}},\]
	where $\alpha(t)={\rm ord}_x(N_{\sigma}(D-tC)|_C)$ and $\beta_t=\alpha(t)+(P_{\sigma}(D-tC)\cdot C)$. 
	\smallskip
	
	In the infinitesimal case, assume $x\not\in{\bf B}_-(D)$. Consider a flag $\overline X\supset E\supset\{x'\}$ where $x'\not\in{\rm Supp}\, N_{\sigma}(\pi^*D-\mu(D;x) E)$. Then $\alpha(t)=0$ for all $0\leq t<\mu(D;x)$ and $\beta(t)=(P_\sigma(D_t)\cdot E)$.
\end{proposition}

\subsubsection{A strategy for computing infinitesimal NO bodies on threefolds}

%In \cite{FLjac} we prove the following:

\begin{proposition}\label{prop:iNOthreefold}
	Let $X$ be a smooth projective projective threefold with $L$ big on $X$. Let $x\in X$ with $x\not\in{\bf B}_-(L)$. Assume that $\rho:\widetilde X\to\overline X$ is a birational morphism with $\widetilde X$ smooth such that
	\begin{enumerate}
		\item $P_{\sigma}(\rho^*L_t)$ is nef for all $0\leq t<\mu=\mu(L;x)$, i.e., $\rho$ determines a simultaneous birational Fujita--Zariski decomposition for all $L_t$.
		\smallskip
		
		\item $(P_{\sigma}(\rho^*L_t)^2\cdot \widetilde E)>0$ for $0<t<\mu$. Together with $(1)$, this condition is equivalent to the strict transform $\widetilde E$ not being a component of ${\bf B}_+(P_{\sigma}(\rho^*L_t))$, and to $P_{\sigma}(\rho^*L_t)|_{\widetilde E}$ being big, both in the range $0<t<\mu$.
	\end{enumerate}
	Then
	\[\Delta_x(L)=\biggl\{(t,x,y)\ \mid\ \begin{cases}0\leq t\leq \mu(L;x)\\ 0\leq x\leq \sup\{s\ \mid\ (P_{\sigma}(\rho^*L_t)+s\rho^* E)|_{\widetilde E}\text{ is big}\}\\ 0\leq y\leq (P_{\sigma}((P_{\sigma}(\rho^*L_t)+x\rho^* E)|_{\widetilde E})\cdot(-\rho^* E|_{\widetilde E}))\end{cases}\biggr\}\]
\end{proposition}

\begin{proof}
	Fix $\overline Y_{\bullet}$ a general linear flag in $E\simeq\bb P^{2}$ on $\overline X$, and let $\widetilde Y_{\bullet}$ be its strict transform in $\widetilde X$. This is well-defined when $\overline Y_3$ is a point in the isomorphism locus of $\rho$. We have $\Delta_x(L)=\Delta_{\widetilde Y_{\bullet}}(\rho^*\pi^*L)$.
	
	By continuity of slices in a convex body, it is sufficient to work in the range where $L_t$ is big, which is $t\in(0,\mu)$. For such $t$, the negative parts $N_{\sigma}(\rho^*L_t)$ vary continuously with $t$. Furthermore only finitely many reduced divisors appear in their support since $N_{\sigma}(\rho^*L_t)\leq N_{\sigma}(\rho^*L_{\mu})+(\mu-t)\rho^*E$. By the assumption that $x\not\in{\bf B}_-(L)$, the strict transform $\widetilde E$ is not a component of $N_{\sigma}(\rho^*L_t)$. For a general choice of a point $\overline Y_3\in E$, then for all $t\in(0,\mu)$ the point $\widetilde Y_3$ is not in ${\rm Supp}(N_{\sigma}(\rho^*L_t))$ and $\overline Y_3$ is not in ${\rm Supp}(N_{\sigma}(L_t))$. It follows that $\Delta_{\overline Y_{\bullet}}(L_t)=\Delta_{\widetilde Y_{\bullet}}(\rho^*L_t)=\Delta_{\widetilde Y_{\bullet}}(P_{\sigma}(\rho^*L_t))$.
	
	We have $\Delta_x(L)_{\nu_1\geq t}=\Delta_{\overline Y_{\bullet}}(\pi^*L)_{\nu_1\geq t}$. Up to translation by $(t,0,0)$, this is $\Delta_{\overline Y_{\bullet}}(L_t)=\Delta_{\widetilde Y_{\bullet}}(\rho^*L_t)=\Delta_{\widetilde Y_{\bullet}}(P_{\sigma}(\rho^*L_t))$. Then $\Delta_x(L)_{\nu_1=t}=\Delta_{\widetilde Y_{\bullet}}(P_{\sigma}(\rho^*L_t))_{\nu_1=0}=\Delta_{\widetilde X|\widetilde E}(P_{\sigma}(\rho^*L_t))=\Delta_{\widetilde Y_{\bullet}}(P_{\sigma}(\rho^*L_t)|_{\widetilde E})$. In the surface $\widetilde E$, the curve $\widetilde Y_2$ in the flag $\widetilde Y_{\bullet}$ is the strict transform of a very general line, a divisor of class $-\rho^*E|_{\widetilde E}$. Apply the global (not infinitesimal) case of Proposition \ref{prop:NOsurface}. Note that ${\rm ord}_{\widetilde Y_3}(N_{\sigma}(P_{\sigma}(\rho^*L_t)|_{\widetilde E}-s\widetilde Y_2))=0$ whenever $\widetilde Y_3$ is very general and $P_{\sigma}(\rho^*L_t)|_{\widetilde E}-s\widetilde Y_2$ is pseudoeffective. Indeed $\widetilde Y_2$ is not in the support of any $N_{\sigma}(P_{\sigma}(\rho^*L_t)|_{\widetilde E}-s\widetilde Y_2)$ since $\widetilde Y_2$ is nef on $\widetilde E$. Furthermore as $D$ ranges through pseudoeffective divisors on $\widetilde E$, the supports of $N_{\sigma}(D)$ range through a countable set of closed reduced divisors.
\end{proof}

\section{Infinitesimal picture for two surfaces}

On surfaces Zariski decompositions exist without the need for further blow-ups, and iNObodies are computable by Proposition \ref{prop:NOsurface}. We present here two relevant cases that appear as components of our 3-dimensional products.

\subsection{Products of curves}

\begin{lemma}[Dimension 2]\label{lem:twoelliptic}
	%Let $A=E_1\times E_2$ be the product of elliptic curves. The principal polarization that we will consider is given by the line bundle $\mathcal{L}=\sO_A(1,1)$. In this case there are two Seshadri curves passing through the origin for the line bundle $\mathcal{L}$ given by each fiber, and using the behaviour of the asymptotic multiplicity on the blow-up, work due to Nakamaye \cite{N96} or \cite{N05}, it is not hard to find the full description of the following Newton-Okounkov body
	%\[
	%\Delta_{Y_{\bullet}}(\mathcal{L}) \ = \ \{(t_1,t_2)\in \RR^2_+ \ | \ t_2\leq \textup{min}(t_1, 2-t_1)\} \ = \ \textup{convex hull}\{(0,0),(1,1),(2,0)\}
	%\]
	%where $Y_{\bullet}$ sits on the blow-up of the origin and consists of the exceptional divisor and a general point.
	
	%The ideas here are as below, but in a much simpler fashion. Basically through the origin there are two corresponding fibers $C_1,C_2$ that are the Seshadri curves for the line bundle $\mathcal{L}$, giving us $\epsilon(\mathcal{L};0)=1$. Now for any $t\geq 1$, by Nakamaye we have 
	%\[
	%\mult_{\overline{C}_i}\big(||B_t||\big) \ \geq \ t-1 \ ,
	%\]
	%where $B_t=\pi^*(\mathcal{L})-tE$ and $\overline{C}_i$ is the proper transform of$C_i$. In terms of how the Newton-Okounkov polygon is computed on surfaces and using these lower bounds on multiplicity with respect to these negative curves and taking into account the Zariski decomposition, then the area $\Delta_{Y_\bullet}(\mathcal{L})\cap [1,\infty)\times \RR$ is below the line connecting the point $(1,1)$ and $(2,0)$. See for example \cite[Proposotion 4.2]{KL18}.
	%\end{example}

	%\begin{lemma}
	Let $C_1,C_2$ be smooth projective curves. Let $x=(c_1,c_2)\in C_1\times C_2$. Denote by $f_1$ the fiber $\{c_1\}\times C_2$ and by $f_2$ the fiber $C_1\times\{c_2\}$. Let $\pi:X\to C_1\times C_2$ be the blow-up of $x$ with exceptional divisor $E$. Denote by $\overline{f_1}$ and $\overline{f_2}$ the strict transforms on $X$ of the fibers through $x$. Fix $d_1\geq d_2> 0$. 
	Then 
	\begin{enumerate}
		\item $\mu(d_1f_1+d_2f_2;x)=d_1+d_2$ and $\epsilon(d_1f_1+d_2f_2;x)=d_2$.
		\smallskip
		
		\item For $0\leq t\leq d_1+d_2$ put 
		\[L_t\deq\pi^*(d_1f_1+d_2f_2)-tE\quad\text{ and }\quad\tau\deq d_1+d_2-t.\] Then  
		\[P_{\sigma}(L_t)=\min\{d_1,\tau\}\overline f_1+\min\{d_2,\tau\}\overline f_2+\tau E\] 
		\smallskip
		
		\item $\vol(L_t)=\begin{cases}2d_1d_2-t^2&\text{, if } 0\leq t\leq d_2\\ d_2^2+2d_1d_2-2d_2t&\text{, if }d_2\leq t\leq d_1\\
			(d_1+d_2-t)^2&\text{, if }d_1\leq t\leq d_1+d_2\end{cases}$. 
		\smallskip
		
		\item For a general linear flag $Y_{\bullet}$ in $E$ we have 
		\[
		\inob{x}{d_1f_1+d_2f_2}=\textup{ convex hull }\{(0,0),(d_2,d_2),(d_1,d_2),(d_1+d_2,0)\}
		\]
		is an isosceles trapezoid that degenerates to an isosceles triangle when $d_1=d_2$. 
	\end{enumerate}
\end{lemma}

\begin{proof}
	(1). The strict transforms $\overline{f_1}$ and $\overline{f_2}$ are negative curves that do not intersect. Then $d_1\pi^*f_1+{d_2}\pi^*f_2-(d_1+{d_2})E=d_1\overline{f_1}+{d_2}\overline{f_2}$ is effective and negative (equal to the negative part of its Zariski decomposition), hence not big.
	
For the Seshadri constant computation, the divisor
	$L_{d_2}\coloneqq d_1\pi^*f_1+{d_2}\pi^*f_2-d_2E\equiv {d_1}\pi^*f_1+d_2\overline{f_2}$ is effective. Since $\pi^*f_1$ is nef and $L_{d_2}\cdot\overline{f_2}={d_1}-d_2\geq 0$, the divisor $L_{d_2}$ is nef. It is not ample since $L_{d_2}\cdot\overline{f_1}=0$. In particular $\overline{f_1}$ is a Seshadri curve.
	
(2). When $0\leq t\leq d_2$, then $\tau\geq d_1\geq d_2$ and $L_t=\min\{d_1,\tau\}\overline{f_1}+\min\{d_2,\tau\}\overline{f_2}+\tau E$ is nef by (1).	

When $d_2\leq t\leq {d_1}$, then $t,\tau\in[d_2,d_1]$ and
$L_t=d_1\overline f_1+{d_2}\overline{f_2}+(d_1+{d_2}-t)E=(\tau\overline {f_1}+d_2\overline{f_2}+\tau E)+(t-d_2)\overline {f_1}$. 
	As in $(1)$, we have that $\tau\overline {f_1}+d_2\overline{f_2}+\tau E=\tau\pi^*f_1+d_2\overline{f_2}$ is nef (has intersection $\tau-d_2\geq 0$ with $\overline{f_2}$) and has 0 intersection with the negative curve $\overline {f_1}$.
	
	When ${d_1}\leq t\leq d_1+{d_2}$, then $0\leq\tau\leq d_2\leq d_1$ and $L_t=\tau(\overline {f_1}+\overline{f_2}+E)+((t-{d_2})\overline{f_1}+(t-d_1)\overline{f_2})$. The divisor $\overline{f_1}+\overline{f_2}+E=\pi^*(f_1+f_2)-E$ is nef as in (1), and has 0 intersection with the negative curves $\overline{f_1}$ and $\overline{f_2}$ that do not intersect each other.
	
	(3) follows from (2) since $\vol(L_t)=P_{\sigma}(L_t)^2$. 
	
	(4). For $t<d_1+d_2$, we have $(P_{\sigma}(L_t)\cdot E)>0$ from the computations above, hence $E\not\subset{\bf B}_+(L_t)$. The result is an easy application of Proposition \ref{prop:NOsurface}.
\end{proof}

We also consider a class of examples where the polarization is not of box-product type.

\begin{lemma}
	Let $C$ be a smooth projective curve of genus $2$. 
	Let $X=C\times C$ and $x\in X$ a point not on the diagonal $\Delta$. Let $L=d_1f_1+d_2f_2+\Delta$ with $d_1\geq d_2>1$ so that $L$ is ample. Then
	\begin{enumerate}
		\item $\mu(L;x)=d_1+d_2$.
		\smallskip
		
		\item With $\overline f_i$ and $\tau$ as in the Lemma above, $P_{\sigma}(L_t)=\min\{d_1,\tau+1\}\overline f_1+\min\{d_2,\tau+1\}\overline f_2+\overline\Delta+\tau E$.
		\smallskip
		
		\item $\inob{x}{L}=\textup{convex hull}\left\{(0,0),\ (d_2+1,d_2+1),\  (d_1+1,d_2+1),\ (d_1+d_2,2),\ (d_1+d_2,0)\right\}$.
	\end{enumerate}
\end{lemma}
\begin{proof}
	The strict transforms $\overline f_1$, $\overline f_2$, $\overline\Delta$ and $E$ intersect as follows
	
	\begin{center} 	
		\begin{tabular}{|c|c|c|c|c|}
			\hline & $\overline f_1$&$\overline f_2$&$\overline\Delta$&E\\
			\hline $\overline f_1$&$-1$ &$0$ &$1$ &$1$ \\
			\hline $\overline f_2$& $0$&$-1$ &$1$ &$1$ \\
			\hline $\overline\Delta$& $1$&$1$ &$-2$ &$0$ \\
			\hline $E$&$1$ &$1$ &$0$ &$-1$ \\
			\hline
		\end{tabular}
	\end{center} 
	The class $P\deq \overline f_1+\overline f_2+\overline\Delta$ is nef and pairs trivially against any of its components, in particular $(P^2)=0$ and $P$ is not big. Then $L_{d_1+d_2}=P+((d_1-1)\overline f_1+(d_2-1)\overline f_2)$ is a Zariski decomposition, thus $\vol(L_{d_1+d_2})=(P^2)=0$. This gives part (1).
	The remaining parts are analogous to the previous Lemma.
\end{proof}

\subsection{Genus two Jacobians} 
Let $C$ be a smooth projective curve of genus $2$ and let $J_i\deq\textup{Pic}^i(C)$ such that $J_0$ is the Jacobian surface of $C$. 

\begin{proposition}\label{prop:genus2surface}
Under the notation above, let $\theta$ be the principal polarization on $J_0$ and $\pi:\overline J_0\to J_0$ be the blow-up of the origin with exceptional divisor $E$. Then 
	\begin{enumerate}
		\item $\pi^*\theta-tE$ is nef if $0\leq t\leq\frac 43$, and $\pi^*\theta-tE$ is effective if $0\leq t\leq\frac 32$. The class $4\pi^*\theta-6E$ is represented by the strict transform $\overline R$ of the curve $R=\{\cal O_C(2x-2p)\ \mid\ x\in C\}$ where $p$ is any of the 6 Weierstrass points. Furthermore $\overline R$ is irreducible with negative self-intersection.
		\smallskip
		
		\item The positive part in the Zariski decomposition of $L_t\deq\pi^*\theta-tE$ is
		\[P_{\sigma}(L_t)=\begin{cases}
			\pi^*\theta-tE&\text{, if }0\leq t\leq\frac 43\\
			(9-6t)(\pi^*\theta-\frac 43E)&\text{, if }\frac 43\leq t\leq\frac 32
		\end{cases}\]
	The negative part is supported on $\overline R$.
	\smallskip
	
	\item The infinitesimal Newton--Okounkov body of $\theta$ is the triangle with vertices at $(0,0)$, $(\frac 32,0)$ and $(\frac 43,\frac 43)$. 	
	\end{enumerate}
\end{proposition}
\begin{proof}
The doubling map $J_1\to J_2$ sends $C\subset J_1$ to $R$.  Here $J_i={\rm Pic}^i(C)$. The restriction of the doubling map to $C$ is birational and sends only the 6 Weierstrass points to the point $o$ corresponding to the hyperelliptic pencil. Using that the doubling map is \' etale and the previous observation, $\text{mult}_oR=6$. We have that $C\equiv\theta$ on $J_1$. The pushforward of $C$ is then represented by $R$ and one computes that it has class $4\theta$. The claim on the class of $\overline R$ follows.

The effective cycle $\frac 14\overline R+\frac 16E\equiv\pi^*\theta-\frac 43E$ has 0 intersection with $\overline R$ and positive intersection with $E$. Part (1) follows.
Part (2) follows easily from (1).

For (3) choose a point $x'\in E$ different from the 6 points $\overline R\cap E$. For $0<t<\frac 32$, the curve $E$ is not in ${\bf B}_+(L_t)$. Apply Proposition \ref{prop:NOsurface}.
\end{proof}

\section{Product type threefolds}

We use birational Fujita--Zariski decompositions and Proposition \ref{prop:iNOthreefold} to compute infinitesimal Newton--Okounkov bodies on three examples of threefolds that are products between a curve and a surface.

\subsection{Curve times projective plane}
Let $C$ be a smooth projective curve. Let $L_1$ be a line bundle on $C$ of degree $a>0$. 
Set $X\deq C\times\bb P^2$ polarized by $L\deq L_1\boxtimes\cal O_{\bb P^2}(b)$ where $b>0$. Consider $x=(c_0,p)\in X$ an arbitrary point.
We have the following divisors on $X$:
\begin{itemize}
	\item $f$ the fiber of the first projection over $c_0$, and 
	\item $H$ the preimage of line in $\bb P^2$ through $p$.
\end{itemize}
The classes of $f$ and $H$ generate $N^1(X)_{\bb R}$. Furthermore they generate both the nef and the pseudoeffective cones of $X$.
Consider also the curves 
\begin{itemize} 
\item $C_x\deq H\cap H'$, where $H'$ is defined analogously to $H$ using a different line through $p$. Note that $C_x$ is the preimage of $p$ by the second projection.
\item $\bb P^1_x\deq f\cap H$. 
\end{itemize} 
These curves are a dual basis for $f,H$. They span both the Mori cone and the movable cone of curves on $X$.

Let $\pi:\overline X\to X$ be the blow-up of $x$ with exceptional divisor $E$. Let $\ell$ be a line in $E$. Denote by $\overline f$, etc. the strict transforms in $\overline X$. Note that $\overline f$ is the blow-up of $\bb P^2$ in $p$, while $\overline H$ is the blow-up of $C\times\bb P^1$ in $(c_0,p)$.

\begin{proposition}\label{prop:blCxP2positive}
	With notation as above,
	\begin{enumerate}[(1)]
		\item $\Eff(\overline X)=\langle\overline f,\ \overline H,\ E\rangle$. In particular $\mu(L;x)=a+b$.
		\smallskip
		
		\item $\Mov(\overline X)=\langle \pi^*f,\ \pi^*H,\ \overline H\rangle$.  In particular $\nu(L;x)=b$.
		\smallskip
		
		\item $\Nef(\overline X)=\langle \pi^*f,\ \pi^*H,\ \overline f+\overline H+E\rangle$. In particular $\epsilon(L;x)=\min\{a,b\}$. 
		\smallskip
		
		\item For $t\in[0,a+b)$, we have $P_{\sigma}(L_t)=L_t-(t-b)_+\cdot\overline f=\min\{a,a+b-t\}\cdot\overline f+b\overline H+(a+b-t)E$.
	\end{enumerate}
\end{proposition}
\begin{proof} 
Clearly $N^1(\overline X)_{\bb R}$ is generated by $\pi^*f$, $\pi^*H$, and $E$.

(1). Let $\alpha=a\overline f+b\overline H+cE$ be the class of an irreducible effective divisor on $\overline X$ distinct from $\overline f$, $\overline H$, and $E$. It is sufficient to prove that $a,b,c\geq 0$. Its pushforward to $X$ is also an irreducible effective divisor, which gives $a,b\geq 0$.
Its restriction to $\overline f\simeq{\rm Bl}_p\bb P^2$ has class $bh+(c-a-b)E$ where $h$ is the class of the pullback of a line in $\bb P^2$ and by abuse $E$ is the exceptional line over $p$. This is effective by our assumptions, which implies $c\geq a\geq 0$.

(2). The pullback of a nef divisor class is again nef, and in particular movable. Furthermore $\overline H\equiv\overline H'$. This implies that $\overline H$ is movable. The claimed generators are indeed movable. Conversely, let $\alpha=a\overline f+b\overline H+cE$ be a movable divisor class. It is in particular pseudoeffective, thus $a,b,c\geq 0$. Its restriction to any divisor must also be pseudoeffective. 
The restriction to $E=\bb P^2$ has degree $a+b-c$, thus $a+b\geq c$.
The restriction to $\overline f$, as in $(1)$ gives the condition $c\geq a$.
The cone cut out by the conditions $a+b\geq c\geq a\geq 0$ is the cone in the conclusion. 

(3). The dual of the cone of curves $\langle\overline C_x,\ \overline{\bb P}^1_x,\ \ell\rangle$ is generated by  $\pi^*f$, $\pi^*H$, and $\overline f+\overline H+E$ as one checks directly. It is then sufficient to prove that these three claimed generators are nef. This is clear for the first two. For the last one, which can be written as $\overline f+\pi^*H$, it is sufficient to prove that its restriction to $\overline f$ is nef. This restriction is $h-E$, which is indeed nef. 

Dually, the Mori cone of curves is generated by $\overline C_x$, $\overline{\bb P}^1_x$, and $\ell$. Thus a class $\alpha$ as above is nef if and only if $a+b\geq c\geq\max\{a,b\}$.

(4). Negative parts (e.g., $N_{\sigma}(L_t)$) cannot contain movable components like $\overline H$. Since $L$ is ample, then $x\not\in{\bf B}_+(L)$ and then $E$ is not a component of any ${\bf B}_+(L_t)$ while $L_t$ is big, in particular it is also not in the support of $N_{\sigma}(L_t)$. We deduce that $N_{\sigma}(L_t)=m_t\cdot\overline f$ for some $m_t\geq 0$. 
Then $(P_{\sigma}(L_t))|_{\overline f}=(L_t-m_t\cdot\overline f)|_{\overline f}\equiv b(h-E)+(b-t+m_t)E$ is pseudoeffective which implies $b-t+m_t\geq 0$, i.e., $m_t\geq(t-b)_+$.

It follows easily from (2) that the claimed formula for the positive part is indeed movable. The maximality of the positive part $P_{\sigma}(L_t)$ among movable divisors dominated by $L_t$ implies the reverse inequality $a-m_t\geq\min\{a,a+b-t\}$.  
\end{proof} 

The only restriction to the nefness of $P_{\sigma}(L_t)$ is the curve $\overline C_x$, and this only when $a-t<(b-t)_+$, i.e., $a<\max\{t,b\}$. We will show that birational Fujta--Zariski approximations exist after blowing-up $\overline C_x$. Note that $\overline C_x=\overline H\cap\overline H'$. Furthermore $(\overline H\cdot \overline C_x)=-1$. In particular the normal bundle of $\overline C_x$ in $\overline X$ is split and semistable. Let 
\[\rho:\widetilde X\to\overline X\]
be the blow-up of $\overline C_x$ with exceptional divisor $G$. Then $G\simeq C\times\bb P^1$ and $-G|_G=\cal O_C(c_0)\boxtimes\cal O_{\bb P^1}(1)$ has degree type $(1,1)$. Note that $\Nef(G)=\Eff(G)$ corresponding to the nonnegative degree types. Furthermore the strict transforms $\widetilde f$ and $\widetilde H$ are isomorphic to $\overline f$ and $\overline H$ respectively, while numerically $\widetilde f\equiv\rho^*\overline f$ and $\widetilde H\equiv\rho^*\overline H-G$. The strict transform $\widetilde E$ is the blow-up of $\bb P^2$ in one point, and numerically $\widetilde E\equiv\rho^*E$.

\begin{proposition}\label{prop:CxP2fujita}
	With the notation above, for all $t\in[0,a+b)$ we have
	\[P_{\sigma}(\rho^*L_t)=\rho^*P_{\sigma}(L_t)-(t-a)_+\cdot G\ .\]
	Furthermore this is nef, in particular we have birational Fujita--Zariski decompositions for all $L_t$ on $\widetilde X$. Consequently
	\[\vol(L_t)=3\min\{a,a+b-t\}\cdot b^2-\min\{t,b\}^3+(t-a)_+^3\ .\]
\end{proposition}
\begin{proof}
	Since $P_{\sigma}(L_t)$ is movable, the negative part $N_{\sigma}(\rho^*P_{\sigma}(L_t))$ is supported on the exceptional divisor $G$ of $\rho$. Let $N_{\sigma}(\rho^*P_{\sigma}(L_t))=uG$. Then $(\rho^*(P_{\sigma}(L_t))-uG)|_G$ is pseudoeffective. Its degree type on $C\times\bb P^1$ is $(a-t+u,u)$. This gives $u\geq(t-a)_+$.
	To prove equality it is sufficient to prove that $\rho^*P_{\sigma}(L_t)-(t-a)_+G$ is movable. In fact we prove that it is nef.
	
	We have $\rho^*P_{\sigma}(L_t)-(t-a)_+\cdot G=\min\{a,a+b-t\}\cdot \widetilde f+b\widetilde H+(a+b-t)\widetilde E+(b-(t-a)_+)\cdot G=\min\{a,a+b-t\}\cdot \widetilde f+b\widetilde H+(a+b-t)\widetilde E+\min\{b,a+b-t\}\cdot G=\min\{a,a+b-t\}\cdot \rho^*\pi^*f+b\widetilde H+(b-t)_+\cdot \widetilde E+\min\{b,a+b-t\}\cdot G$. The class $f$ is nef on $X$. It is enough to check that the restrictions to the components $\widetilde H$, $\widetilde E$, $G$ are nef. 
	
	Given that on $G$ nefness and pseudoeffectivity are equivalent, we have already checked the restriction to $G$. On $\widetilde E$, which is the blow-up of $\bb P^2$ in a point, there are two test curves for nefness: the exceptional line and the strict transform of a line in $\bb P^2$ through the point that we blow-up. These are also contained in $G$ and $\widetilde H$. It is thus sufficient to prove that the restriction to $\widetilde H$ is nef. Recall that $\widetilde H\simeq\overline H$ is the blow-up of of $C\times\bb P^1$ in a point. By some abuse, denote by $\pi$ the blow-up map and by $E$ its exceptional divisor. Denote also by $f_1,f_2$ the classes of fibers of the two projections on $C\times\bb P^1$, and by $\overline f_1,\overline f_2$ their strict transforms in $\overline H$. The restriction class we are interested in works out to 
	\[\min\{a,a+b-t\}\cdot \pi^*f_1+\min\{b,a+b-t\}\cdot\pi^*f_2-\min\{t,a,b,a+b-t\}\cdot E\ .\]
	Its nefness is tested by intersecting with $\overline f_1,\overline f_2,E$. These intersections are $\min\{b,a+b-t\}-\min\{t,a,b,a+b-t\}$, $\min\{a,a+b-t\}-\min\{t,a,b,a+b-t\}$, and respectively $\min\{t,a,b,a+b-t\}$. These are clearly nonnegative. 	
	
	For the volume computation, we have $\vol(L_t)=\vol(\rho^*L_t)=\vol(P_{\sigma}(\rho^*L_t))=(P_{\sigma}(\rho^*L_t)^3)$ since the positive part is nef. For the computation of the intersection number, write $P_{\sigma}(L_t)=\min\{a,a+b-t\}\cdot\pi^*f+b\pi^*H-\min\{t,b\}E$ and use that  $(P_{\sigma}(\rho^*L_t)^3)=((\rho^*P_{\sigma}(L_t)-(t-a)_+\cdot G)^3)=(\rho^*P_{\sigma}(L_t)^3)+3(t-a)_+^2\cdot(\rho^*P_{\sigma}(L_t)\cdot G^2)-(t-a)_+^3(G^3)$. Intersections of form $(\rho^*\beta\cdot G)$ are 0 whenever $\beta$ is a curve class (like $(P_{\sigma}(L_t)^2)$) because $G$ is the blow-up of a locus of codimension at least $2$. Similarly intersections of form $(\pi^*\gamma\cdot E)$ on $\overline X$ are 0 whenever $\gamma$ is a class of positive codimension. Since $-G|_G$ is a divisor class of degree type $(1,1)$ on $G\simeq C\times\bb P^1$, we have $(\rho^*P_{\sigma}(L_t)\cdot G^2)=-(L_t\cdot\overline C_x)=t-a$ and $(G^3)=2$. The computation follows easily. 
\end{proof}

\begin{theorem}\label{thm:CP}Let $C$ be a smooth projective curve and let $L_1$ be a line bundle of degree $a>0$ on $C$. Let $b>0$ and $L=L_1\boxtimes\cal O_{\bb P^2}(b)$ on $X=C\times\bb P^2$. Let $x\in X$ be an arbitrary point. Then
	\begin{enumerate}[(1)]
		\item $\inob{x}{L}$ is the convex hull of the following sets of vertices
		\[\begin{cases}(0,0,0),\ (a+b,0,0),\ (a,b,0),\ (b,b,0),\ (b,0,b),\ (a+b,0,b)&\text{ , if }a\geq b\\
			(0,0,0),\ (2a,0,0),\ (a,a,0),\ (a,0,a),\ (2a,0,a)&\text{, if }a=b\\
		(0,0,0),\ (a+b,0,0),\ (b,a,0),\ (a,a,0),\ (b,0,b),\ (a+b,0,b),\ (b,a,b-a)	&\text{, if }a\leq b\end{cases}\]
	See Figures \ref{fig:CP32}, \ref{fig:CP11}, and \ref{fig:CP23}.
	\smallskip
	
		\item Let $a=1-s$ and $b=s$ for $0<s<1$. Also denote by $L_s$ the corresponding box-product divisor of degree type $(1-s,s)$. Then in $\bb R^4_+$, the bodies $\{s\}\times\inob{x}{L_s}$ glue to the convex hull of the 6 vertices $(0,0,0,0)$, $(1,0,0,0)$, $(0,1,0,0)$, $(1,1,0,0)$, $(1,1,0,1)$, $(\frac 12,\frac 12,\frac 12,0)$.
	\end{enumerate}
\end{theorem}

\begin{figure}
	\begin{minipage}{.33\textwidth}
		\centering
		\includegraphics[width=.9\linewidth]{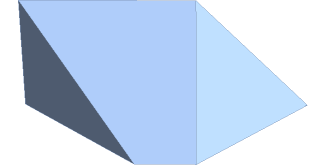}
		\captionof{figure}{$(a,b)=(3,2)$}
		\label{fig:CP32}
	\end{minipage}%
	\begin{minipage}{.33\textwidth}
		\centering
		\includegraphics[width=1.1\linewidth]{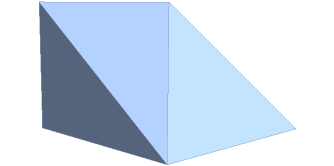}
		\captionof{figure}{$a=b$}
		\label{fig:CP11}
	\end{minipage}
\begin{minipage}{.33\textwidth}
	\centering
	\includegraphics[width=.9\linewidth]{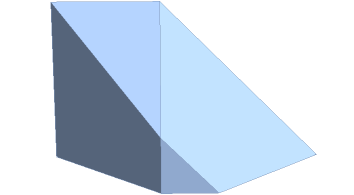}
	\captionof{figure}{$(a,b)=(2,3)$}
	\label{fig:CP23}
\end{minipage}
\end{figure}

\begin{proof}
(1).	The plan is to apply Proposition \ref{prop:iNOthreefold}. Hypothesis (1) holds due Proposition \ref{prop:CxP2fujita}. For (2) we want to check that if $0<t<a+b$, then $(P_{\sigma}(\rho^*L_t)^2\cdot\widetilde E)>0$. Recall that $\widetilde E$ is the blow-up of $\bb P^2$ in a point. Then $P_{\sigma}(\rho^*L_t)|_{\widetilde E}$ is numerically $\min\{t,b\}\cdot h-(t-a)_+\cdot E$ with the notation in the proof of Proposition \ref{prop:blCxP2positive}.(1). The assumption $0<t<a+b$ implies that this class is nef and big.
	
	From Proposition \ref{prop:iNOthreefold} we conclude that $\inob{x}{L}$ is determined by the Zariski decompositions of $(P_{\sigma}(\rho^*L_t)+x\rho^*E)|_{\widetilde E}$. The restriction is 
	\[(\min\{t,b\}-x)\cdot h-(t-a)_+\cdot E\ .\]
	This is already nef for $x\in[0,\min\{t,b\}-(t-a)_+]=[0,\min\{t,a,b,a+b-t\}]$ and not pseudoeffective when $x>\min\{t,a,b,a+b-t\}$. When $x$ is in the nef interval above, then \[(P_{\sigma}((P_{\sigma}(\rho^*L_t)+x\rho^*E)|_{\widetilde E})\cdot(-\rho^*E|_{\widetilde E}))=\min\{t,b\}-x\ .\]
	Thus $\inob{x}{L}$ is the graph of $(t,x)\mapsto\min\{t,b\}-x$ over $0\leq t\leq a+b$ and $0\leq x\leq\min\{t,a,b,a+b-t\}$.
	
	The coordinates $(t,x)$ above range in the isosceles trapezoid with vertices $(0,0)$, $(a+b,0)$, $(\max\{a,b\},\min\{a,b\})$, $(\min\{a,b\},\min\{a,b\})$. Over it, the function $\min\{t,b\}-x$ has two branches that split the trapezoid into the regions $t\leq b$ where $\inob{x}{L}$ in $(t,x,y)$ coordinates is the volume below the graph of $y=t-x$, and the region $t\geq b$ where the body is the volume below the graph of $y=b-x$. To find the vertices of the body, we plug in the vertices of these two regions of the trapezoid into the corresponding branch for $y$.
	
(2). We used \cite{Polymake} to find the vertices of the polytope cut out by the inequalities above.	
\end{proof}

\subsection{Products of curves}

\begin{theorem}\label{thm:CCCfinalform}
	Let $X=C_1\times C_2\times C_3$ be a product of smooth projective curves. Let $x=(c_1,c_2,c_3)\in X$. Let ${\rm pr}_i:X\to C_i$ denote the corresponding projection with fiber class $f_i$, and let $L=d_1f_1+d_2f_2+d_3f_3$ with $d_1\geq d_2\geq d_3>0$. Then
	\begin{enumerate}
		\item $\mu(L;x)=d_1+d_2+d_3$.
		\smallskip
		
		\item Let $\pi:\overline X\to X$ be the blow-up of $x$ with exceptional divisor $E$. 
		For $0\leq t\leq d_1+d_2+d_3$, put 
		\[L_t\deq\pi^*L-tE\quad\text{ and }\quad\tau\deq d_1+d_2+d_3-t\quad\text{ and }\quad n_i\deq\min\{d_i,\tau\}\ .\] Let $\overline f_i$ denote the strict transform of the fiber ${\rm pr}_i^{-1}\{c_i\}$. Then  
		\[P_\sigma(L_t)=\sum\nolimits_{i=1}^3n_i\overline f_i+\tau E.\]
		
		\item Let $\overline f_{ij}\coloneqq\overline f_i\cap\overline f_j$.  Let $\rho:\widetilde X\to\overline X$ denote the blow-up of the three curves $\overline f_{ij}$ with corresponding exceptional divisors $E_k\simeq C_k\times\bb P^1$.  Denote by $\widetilde f_i$ and $\widetilde E$ the strict transforms of $\overline f_i$ and $E$ respectively. 
		Then 
		\[P_{\sigma}(\rho^*L_t)=\sum\nolimits_{i=1}^3n_i\widetilde f_i+\sum\nolimits_{i=1}^3m_iE_i+\tau\widetilde E\]
		is a nef divisor class, where $m_i\deq \min\{n_j+n_k,\tau\}$. 		 In particular, when $d_1=d_2=d_3=1$, then
		 \[P_{\sigma}(\rho^*L_t)=\begin{cases}
		 	\sum_i\widetilde f_i+2\sum_iE_i+(3-t)\widetilde E&\text{, if }0\leq t\leq 1\\
		 	\sum_i\widetilde f_i+(3-t)\sum_iE_i+(3-t)\widetilde E&\text{, if }1\leq t\leq 2\\
		 	(3-t)(\sum_i\widetilde f_i+\sum_iE_i+\widetilde E)&\text{, if }2\leq t\leq 3
		 \end{cases}\]
	 \smallskip
	 
		\item $\vol(L_t)=6n_1n_2n_3-(\sum_in_i-\tau)^3-3\sum_in_i(n_j+n_k-m_i)^2+3(\sum_in_i-\tau)\sum_i(n_j+n_k-m_i)^2-2\sum_i(n_j+n_k-m_i)^3$. In particular, if $d_1=d_2=d_3=1$, then
		\[\vol(L_t)=\begin{cases}6-t^3&\text{, if }0\leq t\leq 1\\
			2t^3-9t^2+9t+3&\text{, if }1\leq t\leq 2\\
			(3-t)^3&\text{, if }2\leq t\leq 3
			\end{cases}
			\]
			\smallskip
			
		\item $\inob{x}{L}$ is the convex hull of the nine points $(0,0,0)$, $(d_3,d_3,0)$, $(d_2+d_3,0,d_2+d_3)$, $(d_1+d_2+d_3,0,0)$, $(d_1+d_3,0,d_2+d_3)$, $(d_1+d_2,0,2d_3)$, $(d_2,d_3,d_2-d_3)$, $(d_1,d_3,d_2-d_3)$, $(d_1+d_2-d_3,d_3,0)$.	In particular, if $d_1=d_2=d_3=1$, then 
		\[\inob{x}{L}=\textup{convex hull}\bigl((0,0,0),(3,0,0),(1,1,0),(2,0,2)\bigr).\]
	\end{enumerate}
\end{theorem}

Figures \ref{fig:C1C2C3111} and \ref{fig:C1C2C3234} show two cases of the convex body computation. 

\begin{figure}
	\begin{minipage}{.5\textwidth}
		\centering
		\includegraphics[width=.9\linewidth]{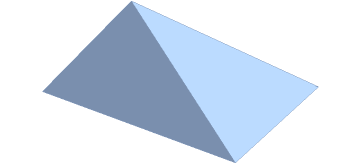}
		\captionof{figure}{$d_1=d_2=d_3=1$}
		\label{fig:C1C2C3111}
	\end{minipage}%
	\begin{minipage}{.5\textwidth}
		\centering
		\includegraphics[width=.8\linewidth]{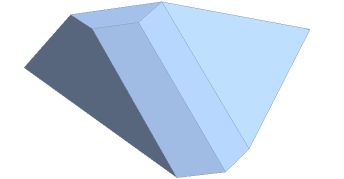}
		\captionof{figure}{$d_1=2,d_2=3,d_3=4$}
		\label{fig:C1C2C3234}
	\end{minipage}
\end{figure}

\begin{proof}
	(1) The strict transform $\overline f_i$ is isomorphic to the blow-up of $C_j\times C_k$ in $(c_j,c_k)$. The divisors $\pi^*f_i,\pi^*f_j,\pi^*f_k, E$ restrict as $0$, the pullbacks of $f_j,f_k$, and as the exceptional $\bb P^1$ respectively.
	
	Clearly $L_t\coloneqq \pi^*L-tE$ is effective for $0\leq t\leq d_1+d_2+d_3$, e.g., $L_{d_1+d_2+d_3}=\sum_i\overline f_i$. 
	For $t>d_i+d_j$, we have that $L_t|_{\overline f_k}$ is not effective by Lemma \ref{lem:twoelliptic}. If $L_t$ is effective, then $L_t-m\overline f_k$ restricts as $L_{t-m}$ on $\overline f_k$. We deduce that $\sigma_{\overline f_k}(L_t)\geq t-d_i-d_j$ and in particular $L_t-(t-d_i-d_j)\overline f_k$ is also effective. If $t>\sum_id_i$, then $L_t-\sum_i(t-d_j-d_k)(\pi^*f_k-E)=(\sum_id_i-t)(\sum_i\pi^*f_i+2E)$ is effective, which is impossible. 
	
	(2) With the arguments in (1), we have	$N_{\sigma}(L_t)\geq \sum_{i=1}^3\max\{0,t-d_j-d_k\}\overline f_i$, where $\{i,j,k\}=\{1,2,3\}$. To verify that equality holds, it remains to prove that $L_t-\sum_i\max\{0,t-d_j-d_k\}\overline f_i=\sum_i\min\{d_i,d_1+d_2+d_3-t\}\overline f_i+(d_1+d_2+d_3-t)E$ is movable. Note that $\tau=d_1+d_2+d_3-t$ is still in $[0,d_1+d_2+d_3]$. We want to see that $\sum_i\min\{d_i,\tau\}\overline f_i+\tau E$ is movable. Since $\overline f_i+E=\pi^*f_i$ is nef and $\min\{d_i,\tau\}\leq\tau$, it is sufficient to prove that $E$ is not a component of the base locus. This follows from the elementary $\sum_i\min\{d_i,\tau\}\geq\tau$ which allows us to write $L_t$ as nonnegative combination of the $\overline f_i$ and general fibers $\pi^*f_i$, all irreducible divisors distinct from $E$. 
	
	(3) Let ${\rm pr}_{ij}:X\to C_i\times C_j$ be the projection. Then $\overline f_{ij}$ is the strict transform of ${\rm pr}_{ij}^{-1}\{(c_i,c_j)\}$ and is isomorphic to $C_k$. We have $\overline f_{ij}\cdot \pi^*f_i=\overline f_{ij}\cdot \pi^*f_j=0$ and $\overline f_{ij}\cdot \pi^*f_k=\overline f_{ij}\cdot E=1$ with the latter scheme theoretically represented by $c_k\in C_k$. In particular $\overline f_{ij}\cdot\overline f_i=\overline f_{ij}\cdot\overline f_j=-1$. Then $\overline f_{ij}$ has normal bundle $\cal O_{C_k}(-c_k)^{\oplus 2}$ and so $E_i=\bb P_{C_i}(\cal O(c_i)^{\oplus 2})\simeq C_i\times\bb P^1$ and $-E_i|_{E_i}=\cal O_{C_i}(c_i)\boxtimes\cal O_{\bb P^1}(1)$ has bidegree $(1,1)$ on $C_i\times\bb P^1$.
	
	We have that $\widetilde f_i=\rho^*\overline f_i-E_j-E_k$. Geometrically it is isomorphic to $\overline f_i$. The three $\widetilde f_i$ are pairwise disjoint. Furthermore, $\widetilde f_i\cap E_i=\emptyset$ and $\widetilde f_i\cap E_j$ sits in $\widetilde f_i$ as $\overline f_k\simeq C_j$ for $i\neq j$ and in $E_j$ as a fiber over a point in $\bb P^1$.
	
	Note that $\widetilde E=\rho^*E$ is the blow-up of $\bb P^2$ in the 3 non-collinear points $p_k\coloneqq\overline f_{ij}\cap E$. Let $e_i\coloneqq E_i\cap\widetilde E$ and $\ell_{jk}\coloneqq\widetilde f_i\cap \widetilde E$. These are 6 lines in the standard Cremona configuration. For instance $\ell_{jk}$ is the strict transform of the line through $p_j$ and $p_k$ and $e_i$ is the exceptional line over $p_i$.
	
	Put $P_t\deq \sum_{i=1}^3n_i\overline f_i+\tau E$. Then $P_t\cdot f_{ij}=\tau-n_i-n_j$. If $n_i+n_j>\tau$, and $0\leq m<n_i+n_j-\tau$, then $(\rho^*P_t-mE_k)|_{E_k}$ is of bidegree $(m+\tau-n_i-n_j,m)$ which is not pseudoeffective. We deduce that $\sigma_{E_k}\rho^*P_t\geq\max\{0,n_i+n_j-\tau\}$. If we prove that $\rho^*P_t-\sum_i\max\{0,n_j+n_k-\tau\}E_i=\sum_{i=1}^3n_i\widetilde f_i+\sum_{i=1}^3m_iE_i+\tau\widetilde E\deq Q_t$ is movable, then it is $P_{\sigma}(\rho^*P_t)$. We verify that it is in fact nef by checking that the restriction to its components $\widetilde f_i$, $E_i$ and $\widetilde E$ is nef.
	
	The restriction $(Q_t)|_{\widetilde E}$ is $\sum_in_i\ell_{jk}+\sum_im_ie_i-\tau h$, where $h$ is the pullback of the line class from $\bb P^2$. Since $\widetilde E$ is a toric variety, its effective cone is generated by $T$-invariant lines, in this case by the $\ell_{ij}$ and by the $e_i$. Nefness is tested by intersecting with these line classes. The intersection with $e_i$ is $n_j+n_k-m_i$, nonnegative by the definition of $m_i$. The intersection with $\ell_{ij}$ is $-n_k+m_i+m_j-\tau$. 
	From (2), we know that $n_1+n_2+n_3\geq\tau\geq n_1\geq n_2\geq n_3\geq 0$. WLOG $i<j$. Then $n_j+n_k\leq n_i+n_k$. If $\tau<n_j+n_k$, then $m_i+m_j-n_k-\tau=\tau-n_k\geq 0$. If $n_j+n_k\leq \tau< n_i+n_k$, then the expression is $n_j\geq 0$. If $n_j+n_k\leq n_i+n_k\leq\tau$, the expression is $n_1+n_2+n_3-\tau\geq0$. Thus the restriction to $\widetilde E$ is nef.
	
	The restriction to $E_i$ has bidegree $(\tau-m_i,n_j+n_k-m_i)$, which is nef by the definition of $m_i$. Nefness on $\widetilde f_i$ is tested on the curves $\widetilde f_i\cap E_j$ and $\widetilde f_i\cap E_k$ by Lemma \ref{lem:twoelliptic}. We have already verified these when restricting on the $E_i$. We conclude that $P_t$ is nef.
	
	(4) Put $Q_t\coloneqq\sum_{i=1}^3n_i\widetilde f_i+\sum_{i=1}^3m_iE_i+\tau\widetilde E$. By the birational invariance of volume and since the movable parts capture all sections, we have $\vol(L_t)=\vol(P_{\sigma}(L_t))=\vol(P_t)=\vol(\rho^*P_t)=\vol(P_{\sigma}(\rho^*P_t))=\vol(Q_t)$. Since $Q_t$ is nef, we have $\vol(Q_t)=(Q_t^3)$ on $\widetilde X$. We have $Q_t=\sum_in_i(\rho^*\pi^*f_i-\rho^*E-E_j-E_k)+\sum_im_iE_i+\tau\widetilde E=\rho^*(\sum_in_i\pi^*f_i-(\sum_in_i-\tau)E)-\sum_i(n_j+n_k-m_i)E_i$.
	 Then from the projection formula, from $\rho_*(E_i^2)=-\overline f_{jk}$, $(E_i^3)=2$ and $(E^3)=1$, we have $(Q_t^3)=(\sum_in_if_i)^3-(\sum_in_i-\tau)^3E^3-3(\sum_in_i\pi^*f_i-(\sum_in_i-\tau)E)\sum_i(n_j+n_k-m_i)^2\overline f_{jk}-\sum_i(n_j+n_k-m_i)^3E_i^3=6n_1n_2n_3-(\sum_in_i-\tau)^3-3\sum_in_i(n_j+n_k-m_i)^2+3(\sum_in_i-\tau)\sum_i(n_j+n_k-m_i)^2-2\sum_i(n_j+n_k-m_i)^3$. 
	 
	 (5) We have already observed that $Q_t=P_{\sigma}(\rho^*P_t)=P_{\sigma}(\rho^*L_t)$ is nef. Furthermore \[F_t\deq (Q_t)|_{\widetilde E}=\sum_in_i\ell_{jk}+\sum_im_ie_i-\tau h.\]
	Looking to apply Proposition \ref{prop:iNOthreefold}, we want to prove $(F_t^2)>0$ for $0<t<d_1+d_2+d_2$ (equivalently $0<\tau<d_1+d_2+d_3$), i.e., $F_t$ is big, but also to understand the Zariski decompositions of $F_t-xh$.
	
	On the toric surface $\widetilde E$, the effective cone is generated by the $e_i$ and $\ell_{jk}$. Dually, the nef cone is generated by $h$, $h'\deq 2h-e_1-e_2-e_3$ (the pullback of the line class by the other 3-point blow-up of $\bb P^2$ structure on $\widetilde E$ coming from the Cremona transform where the exceptional curves are the $\ell_{ij}$), and $h-e_i=h'-\ell_{jk}$. An effective class like the nef $F_t$ fails to be in the interior of the effective cone (big), if and only if $(F_t\cdot \alpha)=0$ for at least one of the generators $\alpha$ of $\Nef(\widetilde E)$. We have $(F_t\cdot h)=\sum_in_i-\tau$. If this intersection number is 0, then since $\tau>0$, we have $n_i>0$ for all $i$, and thus necessarily $n_i<\tau$, i.e., $d_i<\tau$ for all $i$, and then $\sum_in_i-\tau=\sum_id_i-\tau=t>0$ by assumption. This is a contradiction. Next, $(F_t\cdot (h-e_i))=(\sum_in_i-\tau)-(n_j+n_k-m_i)=n_i+m_i-\tau=\min\{n_i+n_j+n_k-\tau,n_i\}>0$ as before. Finally, $(F_t\cdot h')=2(\sum_in_i-\tau)-\sum_i(n_j+n_k-m_i)=\sum_im_i-2\tau$. If $m_i=n_j+n_k$ for 2 or 3 values of $i$, then $\sum_im_i-2\tau>0$ as above. Otherwise $m_i=\tau$ for at least 2 indices $i,j$ and then $\sum_im_i-2\tau=m_k>0$. We conclude that $F_t$ is big, thus Proposition \ref{prop:iNOthreefold} applies. 
	
	%Choose a flag $\widetilde Y_{\bullet}$ of subvarieties $\{p\}\subset\widetilde{\Lambda}\subset\widetilde E\subset\widetilde X$ where $p$ identifies to a general point in $\bb P^2$ that we continue to denote $p$, and $\widetilde{\Lambda}$ identifies to a general line $\Lambda$ in $\bb P^2$. Denote by $Y_{\bullet}$ the flag $\{p\}\subset \Lambda\subset E\subset\overline X$. 
	
	%The negative part of $\rho^*P_t$ is supported on $\sum_iE_i$, not on $\widetilde E$, while the negative part of $B_t$ is supported on $\sum_i\overline f_i$ and not on $E$. It follows that $\Delta_{\widetilde Y_{\bullet}}(Q_t)=\Delta_{\widetilde Y_{\bullet}}(\rho^*P_t)=\Delta_{Y_{\bullet}}(P_t)=\Delta_{Y_{\bullet}}(B_t)$. The case $t=0$ gives $\Delta'(L)\subset\bb R_{\geq0}^3$. The slice $\Delta'(L)\cap [t,\infty)\times\bb R$ identifies after translation by $(-t,0,0)$ to $\Delta_{Y_{\bullet}}(B_t)=\Delta_{\widetilde Y_{\bullet}}(Q_t)$. 
	
	%Since $Q_t$ is nef, the 2-dimensional slice $\Delta'(L)\cap\{t\}\times\bb R^2$ identifies to $\Delta(Q_t|_{\widetilde E})$ for the flag $\{p\}\subset\widetilde L$ in $\widetilde E$. Denote $F_t=Q_t|_{\widetilde E}=\sum_in_i\ell_{jk}+\sum_im_ie_i-\tau h$.
	%The slice $\Delta(F_t)\cap\{x\}\times\bb R$ is the segment $\{x\}\times[0,(P_{\sigma}(F_t-x\widetilde L)\cdot \widetilde L)]$ by \cite[Theorem 6.4]{LM09}.
	
	For the Zariski decompositions of $F_t-xh$ we first specialize to $d_i=1$ for all $i$. Put $e\coloneqq\sum_{i=1}^3e_i$. We have \[F_t=\begin{cases}\sum_i\ell_{jk}+2\sum_ie_i+(t-3)h=th&\text{, if }0\leq t\leq 1\\
		\sum_i\ell_{jk}+(3-t)\sum_ie_i+(t-3)h=th-(t-1)e&\text{, if }1\leq t\leq 2\\
		(3-t)(\sum_i\ell_{jk}+\sum_ie_i-h)=(3-t)(2h-e)&\text{, if }2\leq t\leq 3
	\end{cases}\]
In the 2-dimensional slice of $N^1(\widetilde E)$ generated by $h$ and $e$, the effective cone is generated by $\ell\coloneqq \sum_i\ell_{jk}\equiv 3h-2e$ and by $e$. These are each a union of 3 disjoint $-1$ curves, hence negative. The nef cone is generated by $h$ and by $h'=2h-e$.
Then for $a,b\geq 0$ the class $ah-be$ is effective on $\widetilde E$ iff $a\geq\frac 32b$ and nef iff $a\geq 2b$. For the Zariski decompositions we have 
\[ah-be=\begin{cases}
	(ah-be)+0&\text{, if }a\geq 2b\\
	(2a-3b)h'+(2b-a)\ell&\text{, if } 2b\geq a\geq\frac 32b
\end{cases}\]
Then 
\[P_{\sigma}(F_t-xh)=\begin{cases}(t-x)h&\text{, if }0\leq x\leq t\leq 1\\ 
	(t-x)h-(t-1)e&\text{, if } 1\leq t\leq 2\text{ and }0\leq x\leq\min\{t,2-t\}=2-t\\
	(3-t-2x)h'&\text{, if }1\leq t\leq 2\text{ and } 2-t\leq x\leq\min\{t,\frac{3-t}2\}=\frac{3-t}2\\
	(3-t-2x)h'&\text{, if } 2\leq t\leq 3\text{ and }0\leq x\leq\frac{3-t}2
	\end{cases}
\]
In particular 
\[(P_{\sigma}(F_t-xh)\cdot h)=\begin{cases}
	t-x&\text{, if }0\leq t\leq 2\text{ and } 0\leq x\leq\min\{t,2-t\}\\ 
	6-2t-4x&\text{, if }1\leq t\leq 3\text{ and } \max\{0,2-t\}\leq x\leq \frac{3-t}2
\end{cases}.\] 
The first formula is valid over the triangle $(t,x)\in \textup{convex hull}((0,0),(1,1),(2,0))$. The second over the triangle $(t,x)\in \textup{convex hull}((1,1),(2,0),(3,0))$. 
The claim on the infinitesimal Newton--Okounkov body follows. 

We return to the general case for $d_i$. We first prove that $\inob{x}{L}$ is the set of points $(t,x,y)\in\bb R^3_{\geq 0}$ subject to the following conditions
\[\begin{cases}0\leq t\leq\sum_id_i\\
	0\leq x\leq\min\{\sum_in_i-\tau\ ,\ n_i+m_i-\tau, \frac 12\sum_im_i-\tau\}\\
	0\leq y\leq\sum_in_i+\sum_i\min\{-n_i+m_j+m_k-\tau-x,0\}-\tau-x
\end{cases}
\]From the description of the positive cones on $\widetilde E$, we find that $F_t-xh=\sum_in_i\ell_{jk}+\sum_im_ie_i-(\tau+x)h$ is (pseudo)effective iff the intersections with $h,h-e_i,h'$ are nonnegative. These give the conditions on $x$. 

In the effective range of $t$ and $x$, since $F_t$ is nef, the divisor $F_t-xh$ may fail to be nef only along the $\ell_{ij}$. It follows that the positive part of its Zariski decomposition is $F_t-xh+\sum_i\min\{((F_t-xh)\cdot\ell_{jk}),0\}\ell_{jk}$. Its intersection with $h$ gives the upper bound in the condition on $y$.

For the description of the body as a convex hull, one checks that each of the 9 points verifies the inequalities above and finally that the volume of the convex hull is $d_1d_2d_3=\frac 16\vol(L)$. We leave these as an exercise. For the volume computation, note that the body can be partitioned as the union of the tetrahedron from the case $L=d_3\cdot (f_1+f_2+f_3)$, a triangular prism, and a prism over an isosceles trapezoid. 
\end{proof}

\subsection{Curve times Jacobian surface}

 Let $J$ be the Jacobian of a smooth projective curve $D$ of genus 2. Denote $J_i\deq{\rm Pic}^iD$ and let $o\in J_2$ be the class of the hyperelliptic pencil. Let \[\mu:J_1\to J_2\] be the multiplication by 2 map. We have a canonical Abel--Jacobi inclusion $D\subset J_1$ and let 
 \[R\deq\mu(D).\] Denote $\theta$ the class of the principal polarization on $J$. It is represented by $D$ on $J_1$. The curve $R$ has class $4\theta$ and $\mult_oR=6$ as in Proposition \ref{prop:genus2surface}.
Let $C$ be a smooth projective curve and fix $c\in C$. Let 
\[X\deq C\times J_2\] with projections
\[\xymatrix{& X=C\times J_2\ar[dl]_p\ar[dr]^q& \\ C& & J_2}\] 
Let $x\deq (c,o)\in X$. Denote $J_c\deq p^{-1}(c)=\{c\}\times J_2$ and $Y\deq q^{-1}R=C\times R$. 

\subsubsection{The first blow-up}
Let $\pi:\overline X\to X$ be the blow-up of $x$ with exceptional divisor $E$. Consider in $\overline X$ the strict transforms 
\[\overline J_c\,\deq\, \overline{p^{-1}\{c\}}\simeq {\rm Bl}_oJ_2\quad\text{ and }\quad \overline Y\,\deq\,\overline{q^{-1}R}.\] The normalization map of $R$ identifies with $\mu|_D:D\to R$. 
Then the top horizontal map of the diagram below resolves $\overline Y$.  
\[\xymatrix{{\rm Bl}_{{C\times\mu}|_{C\times D}^{-1}(x)}C\times D\ar[r]\ar[d]& \overline Y\ar[d]\\ C\times D\ar[r]_{C\times\mu|_D}& C\times R}\]
Note that $(C\times\mu|_D)^{-1}(x)$ is the reduced union of the 6 points $(c,d_i)$, where the $d_i$ are the Weierstrass points of $D$. Denote by
\[\mu_Z:Z\to\overline Y\]
the map ${\rm Bl}_{\cup_i(c,d_i)}C\times D\to\overline Y$ above.
Let $f$ denote the class of of a fiber of $p$. Then $\overline J_c\equiv\pi^*f-E$, and $\overline Y\equiv 4\pi^*q^*\theta-6E$. We have the following intersections/restrictions/pullbacks:

\begin{center}
\begin{tabular}{|c|c|c|c|c|c|}
	\hline 
	 & $f$& $\theta$& $E$ &$\overline J_c$& $\overline Y$\\ \hline\hline 
	$\overline J_c$&  $0$ & $\pi^*\theta$& $E$& $-E$ & $\overline R_c=4\pi^*\theta-6E$\\
	\hline 
	$Z$& $\pi^*f_C$ & $8\pi^*f_D$& $E\deq\sum_{i=1}^6E_i$ & $\overline D_c=\pi^*f_C-E$ & $32\pi^*f_D-6E$ 	\\
	\hline $E$ & $0$& $0$& $-h$&  $\ell_J$& $6h$\\ \hline\hline 
\end{tabular}
\end{center}
In the table we restrict by the natural pullbacks the classes on the top row to the surfaces on the left column. For instance $\theta$ restricts to $\overline J_c$ via the composition $\overline J_c\hookrightarrow\overline X\to X\to J_2$. We have abused notation for $\pi$ and $E$ in $\overline J_c$ and $Z$. They always denote point blow-ups and the corresponding exceptional divisor. The classes on the row for $\overline J_c$ correspond to the notation in Proposition \ref{prop:genus2surface}. To restrict to $Z$, we pullback via $\mu_Z$. The classes $f_C$ and $f_D$ are the fibers of the projections to $C$ and $D$. The curve 
\[\overline D_c\text{ is the strict transform }\overline{\{c\}\times D}\text{ in }Z.\]
Finally $E$ in $Z$ is the union of the 6 exceptional lines over each $(c,d_i)$. On $E\simeq\bb P^2$ in $\overline X$, we denote by $h$ the class of a line. Here 
\[\ell_J\deq\overline J_c\cap E\text{ is the exceptional line of }\overline J_c.\] The curve 
\[\overline R_c\deq\overline{\{c\}\times R}\subset\overline X\] 
meets it in 6 points 
\[\{p_1,\ldots,p_6\}\deq \overline R_c\cap \ell_J=\overline R_c\cap E\] that can be seen as the images of $d_i$ under the hyperelliptic map. The curve 
\[\overline C_o\deq\overline{C\times\{o\}}\] 
meets $E$ in a point 
\[\{p_7\}=\overline C_o\cap E\] not on $\ell_J$. Then $\overline Y\cap E$ can be more precisely described as the union of the 6 lines $\ell_i$, where
\[\ell_i\text{ joins }p_i\text{ and }p_7\text{ in }\bb P^2\simeq E.\]

\begin{proposition}\label{prop:firstCxJ}With notation as above set $\alpha\deq a\overline J_c+b\overline Y+eE$. Then
	\begin{enumerate}
		\item $\alpha$ is pseudo-effective iff $a,b,e\geq 0$.
		\smallskip
		
		\item $\alpha$ is nef iff $a+6b\geq e\geq\max\{6b,a+\frac 23b\}$. Equivalently, \[\Nef(\overline X)\cap {\rm Span}(\pi^* f,\pi^*q^*\theta,E)=\langle\pi^*f,\ \pi^*q^*\theta,\ 4\pi^*f+3\pi^*q^*\theta-4E\rangle.\]
		\smallskip
		
		\item 
		$\alpha$ is movable iff $a\geq 0$ and $a+6b\geq e\geq\max\{a,\frac 23b\}$. Equivalently, \[\Mov(\overline X)\cap{\rm Span}(\pi^*f,\pi^*q^*\theta,E)=\langle \pi^*f,\pi^*q^*\theta,3\pi^*q^*\theta-4E,2\pi^*f+3\overline Y\rangle.\]
		%Let $\overline\theta\,\deq\,\overline{C\times D}\subset\overline X$. We have inclusions $\langle \pi^*f,\pi^*q^*\theta,\overline\theta,2\overline\theta+\overline J_c\rangle\subseteq \Mov(\overline X)\subseteq \langle \pi^*f,\pi^*q^*\theta,\overline Y\rangle$.
		\smallskip
		
		\item For $s\in(0,1)$ consider the ample divisor $L_s\deq sq^*\theta+(1-s)f$ on $X$. Then
		%\[\mu(L_s;x)=1+\frac 12s\quad\text{ , }\quad\varepsilon(L_s;x)=\begin{cases}\frac 43s&\text{, if }0<s\leq\frac 37\\ 1-s&\text{, if }\frac 37\leq s<1\end{cases} \quad\text{ , }\quad \nu(L_s;x)=\begin{cases}\frac 32s&\text{, if }0<s\leq\frac 67\\ 1+\frac 13s&\text{, if }\frac 67\leq s<1\end{cases}.\]
		
		\[\mu(L_s;x)=1+\frac 12s\quad\text{ , }\quad\varepsilon(L_s;x)=\min\{\frac 43s,1-s\}\quad\text{ , }\quad \nu(L_s;x)=\min\{\frac 32s,1+\frac 13s\}\] 
				\smallskip
				
		\item Consider $L_{s,t}\deq \pi^*L_s-tE$ on $\overline X$ with $0\leq t\leq\mu(L_s;x)=1+\frac 12s$. Then 
		\begin{align*}P_{\sigma}(L_{s,t})&=L_{s,t}-(t-\frac 32s)_+\cdot\overline J_c-\frac 32(t-1-\frac 13s)_+\cdot\overline Y\\ &=\min\{1-s,1+\frac 12s-t\}\cdotp\overline J_c+\frac 32\min\{\frac 16s,1+\frac 12s-t\}\cdotp\overline Y+(1+\frac 12s-t)\cdotp E\end{align*}
	%	\begin{enumerate}
	%		\item If $0<s\leq\frac 67$, then
	%		\[P_{\sigma}(L_{s,t})=\begin{cases}L_{s,t}&\text{, if }0\leq t\leq \frac 32s\\
	%			L_{s,t}-(t-\frac 32s)\overline J_c&\text{, if }\frac 32s\leq t\leq 1+\frac 13s\\
	%			L_{s,t}-(t-\frac 32s)\overline J_c-\frac32(t-1-\frac 13s)\overline Y&\text{, if } 1+\frac 13s\leq t\leq 1+\frac 12s
	%		\end{cases}\]
	%		\item If $\frac 67\leq s<1$, then
	%		\[P_{\sigma}(L_{s,t})=\begin{cases}L_{s,t}&\text{, if }0\leq t\leq 1+\frac13s\\
	%			L_{s,t}-\frac 32(t-1-\frac 13s)\overline Y&\text{, if }1+\frac 13s\leq t\leq \frac 32s\\
	%			L_{s,t}-(t-\frac 32s)\overline J_c-\frac32(t-1-\frac 13s)\overline Y&\text{, if }\frac 32s\leq t\leq 1+\frac 12s
	%		\end{cases}\]
	%		\end{enumerate}
	\end{enumerate}
\end{proposition}

Figure \ref{fig:CxJ} shows an illustrative slice in the cones $\Eff(\overline X)\supseteq\Mov(\overline X)\supseteq\Nef(\overline X)$.

\begin{figure}
	\centering
	\includegraphics[scale=0.25]{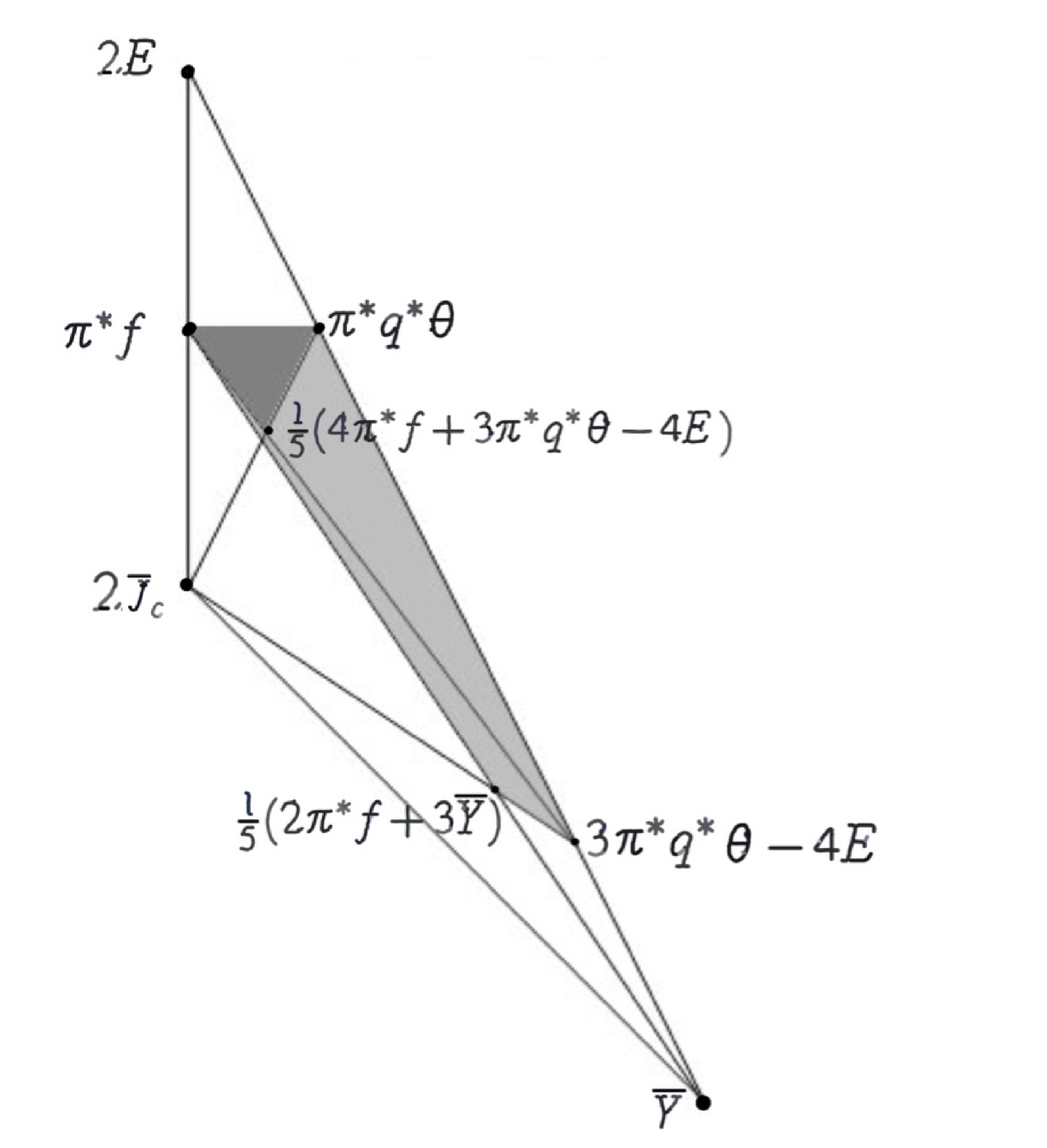}
	\caption{A slice in the cones of divisors on $\overline X$}
	\label{fig:CxJ}
\end{figure}

\begin{proof}
	(1) Note that $\alpha=\pi^*(af+4bq^*\theta)+(e-a-6b)E$. By projecting to $J_2$ we see that $a\geq 0$. After intersecting with a general fiber of $p$, we see that $b\geq 0$. If $e<0$, then WLOG our class is represented by an effective divisor without $\overline J_c$ in its support. By restricting to $\overline J_c$, with the notation in Proposition \ref{prop:genus2surface}, we obtain that $4b\pi^*\theta-(a+6b-e)E$ is effective on $\overline J_c$, which is impossible since $-(a-e)<0$. 
	
	(2) We seek to apply the strategy in \ref{subsec:neftest}. By restricting $\alpha$ to $E$, we obtain $a+6b\geq e$.
	This is the condition $\alpha\cdot\ell\geq 0$, where $\ell$ is a line in $E\simeq\bb P^2$.	
	
	By restricting to $\overline J_c$, by Proposition \ref{prop:genus2surface} we obtain $a+\frac 23b\leq e$. This is the condition $\alpha\cdot{\overline R}_c\geq 0$, where $\overline R_c=\overline{\{c\}\times R}$. 
	
	For $\overline Y$, we pullback to its resolution $Z$. This gives  \[\pi^*(af_C+32bf_D)-(a+6b-e)E,\] 
	that we rewrite as $\pi^*f_C+\frac{16}3b\sum_{i=1}^6\overline f_{D,i}+(e-a-\frac 23b)E$, were 
	\[\overline f_{D,i}\deq\overline{C\times\{d_i\}}\subset Z.\] 
	Using $a+\frac 23b\leq e$, this is a nonnegative combination of $\pi^*f_C$, $E$, and the $\overline f_{D,i}$. The necessary additional conditions for nefness are the nonnegativity of intersection with each $\overline f_{D,i}$. This gives $e\geq 6b$. This is the condition $\alpha\cdot\overline C_o\geq 0$, where $\overline C_o=\overline{C\times\{o\}}$.
	
	The cone cut out by the inequalities above is the cone generated by the three classes $\pi^*f,\ \pi^*q^*\theta,\ 4\pi^*f+3\pi^*q^*\theta-4E$, as claimed.
	
	(3) The classes $\pi^*f$ and $\pi^*q^*\theta$ are movable since they are nef. %The class $\overline\theta$ is movable since it is represented by the strict transform of $C\times D$, but also as combination of $E$ and $\overline Y$. The class $2\pi^*f+\overline Y=4\overline\theta+2\overline J_c$ is also movable as it clearly has no base locus.
	On $\overline J_c$, the class $3\pi^*\theta-4E$ is nef by Proposition \ref{prop:genus2surface}.
	We have a rational dominant map $\overline X\dashrightarrow\overline J_c$ that is not a morphism. We can define a pullback of divisors and it preserves movability. It also takes  $3\pi^*\theta-4E\in N^1(\overline J_c)$ to $3\pi^*q^*\theta-4E\in N^1(\overline X)$.
	 Finally, $2\pi^*f+3\overline Y=2\overline J_c+4(3\pi^*q^*\theta-4E)$  is movable by \ref{subsec:movtest}. We have proved that the claimed generators of the slice are indeed movable.
	
	If $\alpha$ is movable on $\overline X$, then it has pseudoeffective restriction to $E$ and $\overline J_c$, and pseudoeffective pullback to $Z$ by \ref{subsec:movcond}. These conditions translate to $a+6b\geq e$ and respectively $e\geq a$ by Proposition \ref{prop:genus2surface}.
	It follows that our slice of interest of the movable cone is contained in $\langle \pi^*f,\pi^*q^*\theta,\overline Y\rangle$.
	
	It remains to prove that no $a(2\pi^*f+3\overline Y)+b(3\pi^*q^*\theta-4E)+c\overline Y$ with $a,b,c>0$ can be movable.
	It suffices to check that its pullback to $Z$ is not pseudoeffective. The pullback to $Z$ is $a(2\pi^*f_C+96\pi^*f_D-18\sum_{i=1}^6E_i)+b(24\pi^*f_D-4\sum_{i=1}^6E_i)+c(32\pi^*f_D-6\sum_{i=1}^6E_i)=a(2\overline D_c+16\sum_{i=1}^6\overline f_{D,i})+b(4\sum_{i=1}^6\overline f_{D,i})+c(6\sum_{i=1}^6\overline f_{D,i}-4\pi^*f_D)=(2a\overline D_c+(16a+4b+6c)\sum_{i=1}^6\overline f_{D,i})-4c\pi^*f_D$. The seven curves $\overline D_c$ and $\overline f_{D,i}$ have negative self-intersection and are pairwise disjoint. Then $2a\overline D_c+(16a+4b+6c)\sum_{i=1}^6\overline f_{D,i}$ is an effective and negative (in the Zariski decomposition sense) class. As such it cannot dominate any nonzero movable class like $4c\pi^*f_D$. Alternatively, one checks that $\pi^*f_C+\sum_{i=1}^6\overline f_{D,i}$ is nef and has negative intersection with the restriction of $\beta$. 
	
	In Figure \ref{fig:CxJ}, the slice of the movable cone is the quadrilateral with vertices the generators described above. For the description by inequalities, the right side is $a\geq 0$ (the side it partially shares with the effective cone), the top side is $a+6b\geq e$ (the side it shares with the nef cone), the left side is $e\geq a$, and the bottom side is the remaining $e\geq\frac 23b$. 
	
	(4) We have 
	\[L_{s,t}=(1-s)\overline J_c+\frac 14s\overline Y+(1+\frac 12s-t)E.\] The formulas for $\mu(L_s;x)$, $\epsilon(L_s;x)$, $\nu(L_s;x)$ follow from $(1)$, $(2)$, $(3)$ respectively.
	
	(5) $L_{s,t}$ is movable for $0\leq t\leq\nu(L_s;x)$. 
	
	If $s\leq\frac 67$, then $\nu(L_s;x)=\frac 32s$ by $(4)$. If $t=\frac 32s+\delta$, with $0<\delta\leq 1+\frac 12s-\frac 32s=1-s$ (to keep $L_{s,t}$ pseudoeffective) then for $0\leq u< \delta$, we have $(L_{s,t}-u\overline J_c)|_{\overline J_c}=s\theta +(u-t)E=s(\theta-\frac 32E)+(u-\delta)E\in N^1(\overline J_c)$ is not pseudoeffective by Proposition \ref{prop:genus2surface}. We deduce that $\sigma_{\overline J_c}(L_{s,t})\geq\delta=t-\frac 32s$. When $\frac 32s\leq t\leq 1+\frac 13s$, then $L_{s,t}-(t-\frac 32s)\overline J_c=(1+\frac 12s-t)\overline J_c+\frac 14s\overline Y+(1+\frac 12s-t)E$ is movable by $(3)$, hence it is the positive part of $L_{s,t}$ in this range.	
	When $t=1+\frac 13s+\delta$ with $0<\delta\leq\frac 16s$ (again to keep $L_{s,t}$ pseudoeffective), then for $0\leq u<\frac 32\delta\leq\frac 14s$, we have that the pullback of $L_{s,t}-(t-\frac 32s)\overline J_c-u\overline Y$ to the resolution of $\overline Y$ is $(\frac 16s-\delta)\overline D_c+\frac 43(s-4u)\sum_{i=1}^6\overline f_{D,i}+(\frac 23u-\delta)E$. Arguing as in $(3)$, we see that this class cannot be pseudoeffective. This implies $\sigma_{\overline Y}(L_{s,t}-(t-\frac 32s)\overline J_c)\geq \frac 32\delta$. The class $L_{s,t}-(t-\frac 32s)\overline J_c-\frac 32\delta\overline Y=(1+\frac 12s-t)(\overline J_c+\frac 32\overline Y+E)$ is a multiple of a generator of $\Mov(\overline X)$ in $(3)$. The claim on the $\sigma$-Zariski decomposition follows.
	
	The argument for $s\geq\frac 67$ is similar, but we first restrict to $\overline Y$ and then to $\overline J_c$.
\end{proof}

Our goal is to find birational Fujita--Zariski decompositions for $L_{s,t}$. 
Up to scaling, we have seen that we find 
\[\alpha\deq 2\pi^*f+3\overline Y\]
 as a (movable) positive part whenever $1+\frac 12s\geq t\geq\max\{\frac 32s,1+\frac 13s\}$. The class $\alpha$ is not nef. It has negative intersection with both
$\overline R_c=\overline{\{c\}\times R}$ and $\overline C_o=\overline{C\times\{o\}}$. 
Thus $\overline X$ is not where we find the desired decompositions.
We first explain how to treat $\alpha$. This will require the blow-up of $\overline C_o$, but also three blow-ups in copies of $D$ over $\overline R_c$. Then we will show that in fact these blow-ups determine birational Fujita--Zariski for all $L_{s,t}$.

\subsubsection{The second blow-up}
Let $\rho:\widetilde X\to\overline X$ be the blow-up of the two disjoint smooth curves $\overline R_c$ and $\overline C_o$ with exceptional divisors $F_R$ and $F_C$ respectively. We have the following strict transforms
\[\widetilde E=\rho^*E,\quad \widetilde J_c=\rho^*\overline J_c-F_R,\quad \widetilde Y=\rho^*\overline Y-F_R-6F_C.\]
For instance the latter holds because $\overline Y$ is smooth at a general point of $\overline R_c$ and has multiplicity 6 at a general point of $\overline C_o$, since $q^{-1}R=C\times R$ has these properties in $X$. 
Geometrically, $\widetilde E$ is the blow-up of $E\simeq\bb P^2$ in the $7$ points $p_i$. Recall that the first 6 points $p_i$ are on the line $\ell_J=E\cap\overline J_c$. We have an isomorphism $\widetilde J_c={\rm Bl}_{\overline R_c}\overline J_c\simeq\overline J_c$ induced by $\rho$. Finally, via the doubling map $\mu:J_1\to J_2$, we see that $\widetilde Y$ is smooth, isomorphic to $Z$.

Indeed $\mu^{-1}R$ is the union of the 16 translates of $D\subset J_1$ by the theta characteristics $T_j$ (translates of the 2-torsion points in $J_0$). Note that the effective $T_j$ are precisely the six $d_i$, with the other 10 being of form $d_i+d_j-d_1$. Any two translates of $D$ meet transversally in 2 points. After blowing-up the 16 $T_j$ in $J_1$, the strict transforms $\overline{D+T_j}$ are disjoint. 

After blowing-up all $(c,T_j)$ in $C\times J_1$, since $\mu$ is \' etale, the surface $\overline Y$ is still singular along the images of the $\overline{C\times D}\cap\overline{C\times(D+T_j)}$. This is a transverse intersection consisting of two disjoint $\overline{C\times\{d_i\}}$ that map to $\overline C_o$. The preimage of $\overline C_o$ is the union of the 16 such curves. The preimage of $\overline R_c$ is the union of the 16 disjoint curves in the previous paragraph. The base change of $\rho$ separates all $\widetilde{C\times (D+T_j)}$ whose union is the \' etale preimage of $\widetilde Y$. The blow-up of $\overline C_o$ would suffice for this.

The curve $\overline R_c$ is the complete intersection of $\overline J_c$ and $\overline Y$. We have $\overline J_C|_{\overline R_c}=-E|_{\overline R_c}$. 
We claim that $E|_{\overline R_c}=\cal O_D(\sum_{i=1}^6d_i)=\omega_D^{\otimes 3}$. To see this, consider the sum map $D\times D\to J_2$. It factors through the blow-up of $o$ such that the graph of the hyperelliptic involution of $D$ maps to the exceptional $\bb P^1$ as the natural $2$-to-$1$ map. The diagonal maps to $\overline R$. The diagonal and the graph of the hyperelliptic involution meet only at the points $(d_i,d_i)$. Next, $\overline Y|_{\overline R_c}=(\overline Y|_{\overline J_c})|_{\overline R_c}=\cal O_{\overline R_c}(\overline R_c)$. By adjunction, $\omega_D=\omega_{\overline J_c}(\overline R_c)|_{\overline R_c}=\cal O_{\overline J_c}(E+\overline R_c)|_{\overline R_c}$, hence $\overline Y|_{\overline R_c}=\omega_D\otimes\cal O_{\overline R_c}(-E)=\omega_D^{\otimes -2}$.
Then $N_{\overline R_c|\overline X}^{\vee}=\omega_D^{\otimes 2}\oplus\omega_D^{\otimes 3}$ and $F_R$ is its projectivization. Furthermore, $-F_R|_{F_R}$ is the Serre $\cal O(1)$.

The curve $\overline C_o$ is a connected component of the transverse complete intersection $\overline{C\times(D+d_1)}\cap\overline{C\times(D+d_2)}$. We have $\overline{C\times(D+d_1)}|_{\overline C_o}=-E|_{\overline C_o}=\cal O_C(-c)$. As above we obtain $F_C\simeq\bb P_C(\cal O_C(c)^{\oplus 2})\simeq C\times\bb P^1$ and $-F_C|_{F_C}$ is the relative Serre $\cal O(1)$.
We list the relevant restrictions:

\begin{center}
	\begin{tabular}{|c|c|c|c|c|c|c|c|}
		\hline 
		& $f$& $\theta$& $\widetilde E=\rho^*E$& $\widetilde J_c$ & $\widetilde Y$& $F_R$& $F_C$\\
		\hline\hline 
		$\widetilde J_c\simeq\overline J_C$ & $0$ & $\pi^*\theta$  & $E$& $-E-\overline R$ & $0$ & $\overline R$ & $0$ \\
		\hline 
		$\widetilde Y\simeq Z$ & $\pi^*f_C$ & $8\pi^*f_D$ & $E=\sum_{i=1}^6E_i$& $0$ & $-\overline D_c-4\pi^*f_D$ & $\overline D_c=\pi^*f_C-E$ & $\sum_{i=1}^6\overline f_{D,i}$ \\
		\hline 
		$\widetilde E$ & $0$ & $0$ & $-\rho^*h$& $\widetilde{\ell}_J$ & $\sum_{i=1}^6\widetilde\ell_i$ & $\sum_{i=1}^6e_i$ & $e_7$ \\
		\hline 
		$F_R$ & $0$  & $8f$ & $6f$& $\xi-6f$ & $\xi-4f$ & $-\xi$ & $0$ \\ 
		\hline 
		$F_C$ & $f$ & $0$ & $f$& $0$ & $6(\xi-f)$ & $0$ & $-\xi$ 	\\
		\hline 	 
	\end{tabular}
\end{center}
Among the less trivial restrictions we have:
$\widetilde Y|_{\widetilde Y}=32\pi^*f_D-6E-\overline D_c-6\sum_{i=1}^6\overline f_{D,i}=-\overline D_c-4\pi^*f_D$; and $\widetilde Y|_{\widetilde E}=6\rho^*h-\sum_{i=1}^6e_i-6e_7=\sum_{i=1}^6\widetilde\ell_i$, where $e_i$ are the exceptional lines in $\widetilde E$ over the $p_i$; and $\widetilde Y|_{F_C}=\sum_{i=1}^6C\times\{p_i\}=6(\xi-f)$ in $F_C\simeq C\times\bb P^1$. In $F_R$ and $F_C$ we have denoted by $\xi$ the class of the Serre $\cal O(1)$, and by $f$ the class of a $\bb P^1$-fiber over $D$ and $C$ respectively. Because of the explicit descriptions of the projectivizations, we have that $\Eff(F_R)=\langle\xi-6f,f\rangle$ and $\Nef(F_R)=\langle\xi-4f,f\rangle$, while $\Eff(F_C)=\Nef(F_C)=\langle\xi-f,f\rangle$. More generally, with arguments as in \cite[Ch.V.\S2]{H77} or \cite{Ful11} we have the following:
\begin{lemma}
	Let $C$ be a smooth projective curve. Let $L_1,L_2$ be line bundles of degrees $d_1\geq d_2$ on $C$. In the ruled surface $X=\bb P(L_1\oplus L_2)$ denote by $\xi$ the class of the Serre $\cal O(1)$, and by $f$ the class of a fiber of the natural projection $X\to C$. Then
	\[\Eff(X)=\langle\xi-d_1f,f\rangle\quad\text{ and }\quad\Nef(X)=\langle \xi-d_2f,f\rangle.\]
\end{lemma}

Let \[\beta\deq \rho^*\alpha-2F_R-16F_C= 2\rho^*\pi^*f+12\rho^*\pi^*q^*\theta-18\rho^*E-2F_R-16F_C.\] 
It will follow from our work that it is the positive part of $\rho^*\alpha$ in its $\sigma$-Zariski decomposition, i.e., $N_{\sigma}(\rho^*\alpha)=2F_R+16F_C$. For now we explain how the coefficients $2$ and $16$ are found.
Since $\alpha$ is movable, the negative part of $\rho^*\alpha$ is at most supported on the exceptional divisors $F_R$ and $F_C$. The class $\rho^*\alpha-uF_R-vF_C$ restricts of $F_R$ as $u\xi-12f$ and to $F_C$ as $v\xi-16f$. These are not pseudoeffective if $0\leq u<2$ and $0\leq v<16$. Thus the coefficients of $F_R$ and $F_C$ in $N_{\sigma}(\rho^*\alpha)$ are at least 2 and 16 respectively.
%It remains to verify that $\beta$ is actually movable. 
%We have $\beta=2\widetilde J_c+ 4(\rho^*(3\pi^*q^*\theta-4E)-4F_C)$. As in the proof of Proposition \ref{prop:firstCxJ}.(3), we can see $3\pi^*q^*\theta-4E$ as a dominant rational pullback of the nef class $3\pi^*\theta-4E$ in $\overline J_c$. Thus we can realize $3\pi^*q^*\theta-4E$ as limit of effective $\bb Q$-divisors with irreducible supports that do not contain $\overline R_c$, and multiplicity limiting to $4$ along $\overline C_o$. Then $\rho^*(3\pi^*q^*\theta-4E)-4F_C$ can be proved to be movable. These imply that ${\rm Supp}\, N_{\sigma}(\beta)\subseteq{\rm Supp}\, N_{\sigma}(\rho^*)$ is contained in $\widetilde J_c\cap(F_R\cup F_C)$. This intersection is 1-dimensional and so $\beta$ is movable. 
 
	We compute $\beta|_{\widetilde J_c}=\overline R_c$ and $\beta|_{F_R}=2(\xi-6f)$. These are effective, but not nef. They have negative intersection with \[\widetilde R\deq \widetilde J_c\cap F_R.\]
	This curve identifies with $\overline R_c$ via $\rho$, but we use the notation $\widetilde R$ when we want to study it in $\widetilde X$. 
	For the other restrictions, we have $\beta|_{\widetilde Y}=0$, $\beta|_{F_C}=16(\xi-f)$, and $\beta|_{\widetilde E}=2 \widetilde{\ell}_J+16(\rho^*h-e_7)$ which are nef. 
	
	Thus $\beta$ only fails to be nef because of $\widetilde R$. Furthermore $\beta|_{F_R}=2(\xi-6f)$ suggests that we further blow-up $\widetilde R$ ``twice''. 

\subsubsection{The third blow-up} Let 
\[\sigma:\widehat X\to\widetilde X\] be the blow-up of $\widetilde R$ with exceptional divisor $G$. The intersections $\widetilde J_c\cdot\widetilde R=\widetilde J_c|_{F_R}^2=(\xi-6f)^2=-2$ and $F_R\cdot\widetilde R=F_R|_{\widetilde J_c}^2=(\overline R)^2=-4$ lead to $G\simeq \bb P(\omega_D\oplus\omega_D^{\otimes 2})$ and $-G|_G$ is the relative $\cal O(1)$. 

The strict transform $\widehat E$ is the blow-up of $\widetilde E$ in the 6 points $e_i\cap\widetilde{\ell}_J$. Denote the exceptional lines by $g_i$. The surfaces $\widetilde J_c$, $\widetilde Y$, $F_R$, $F_C$ are isomorphic to their strict transforms. We also have $\widehat J_c=\sigma^*\widetilde J_c-G$ and $\widehat F_R=\sigma^*F_R-G$, while $\widehat Y=\sigma^*\widetilde Y$, $\widehat F_C=\sigma^*F_C$, and $\widehat E=\sigma^*\widetilde E$.
We have the following restrictions

\begin{center}
	\begin{tabular}{|c|c|c|c|c|c|c|c|c|}
		\hline 
		& $f$& $\theta$& $\widehat E$& $\widehat J_c$ & $\widehat Y$ & $\widehat F_R$& $\widehat F_C$& $G$\\
		\hline\hline 
		$\widehat J_c$ & $0$ & $\pi^*\theta$  & $E$ & $-E-2\overline R$ & $0$ & $0$ & $0$ & $\overline R$ \\
		\hline 
		$\widehat Y$ & $\pi^*f_C$ & $8\pi^*f_D$ & $\sum_{i=1}^6E_i$ & $0$ & $-\overline D_c-4\pi^*f_D$ & $\overline D_c$ & $\sum_{i=1}^6\overline f_{D,i}$ & $0$ \\
		\hline 
		$\widehat E$ & $0$ & $0$ & $-\sigma^*\rho^*h$ & $\widehat{\ell}_J$ & $\sum_{i=1}^6\widehat{\ell}_i$ & $\sum_{i=1}^6\widehat e_i$ & $\widehat e_7$ & $\sum_{i=1}^6g_i$ \\
		\hline 
		$\widehat F_R$ & $0$  & $8f$ & $6f$ & $0$ & $\xi-4f$ & $-2\xi+6f$ & $0$ & $\xi-6f$\\ 
		\hline 
		$\widehat F_C$ & $f$ & $0$ & $f$ & $0$ & $6(\xi-f)$ & $0$ & $-\xi$ & $0$	\\
		\hline 	 
		$G$ & $0$ & $8f$ & $6f$ & $\xi-2f$ & $0$ & $\xi-4f$ & $0$ & $-\xi$\\
		\hline 
	\end{tabular}
\end{center}

Consider the class
\[\gamma\deq \sigma^*\beta-G= \sigma^*(2\rho^*\pi^*f+12\rho^*\pi^*q^*\theta-18\rho^*E-2F_R-16F_C)-G.\]
We compute $\gamma|_{\widehat J_c}=0$, $\gamma|_{\widehat Y}=0$, $\gamma|_{\widehat E}=2\widehat{\ell}_J+\sum_{i=1}^6g_i+16\sigma^*(\pi^*h-e_7)$, $\gamma|_{\widehat F_C}=0$. These are all nef. Furthermore $\gamma|_{\widehat F_R}=\xi-6f$ and $\gamma|_{G}=\xi-4f$ are effective, but are represented by the curve 
\[\widehat R\deq\widehat F_R\cap G\]
which is negative in both $\widehat F_R$ and in $G$. 

\subsubsection{The fourth blow-up}Consider a further blow-up
\[\nu:\check{X}\to\widehat{X}\]
of $\widehat R$ with exceptional divisor $N$. The intersections $\widehat F_R\cdot\widehat R=-2$ and $G\cdot\widehat R=-2$ lead $N\simeq\bb P(\omega_D^{\oplus 2})\simeq D\times\bb P^1$ and $-N|_N$ is the relative $\cal O(1)$ for the projectivization. 

The strict transform $\check E$ is the blow-up of $\widehat E$ in the 6 points $\widehat e_i\cap g_i$. Denote the exceptional lines by $n_i$. The surfaces $\widehat J_c$, $\widehat Y$, $\widehat F_R$, $\widehat F_C$, and $G$ are isomorphic to their strict transforms in $\check X$. We have $\check F_R=\nu^*\widehat F_R-N$ and $\check G=\nu^*G-N$, while the other strict transforms represent the pullbacks. For restrictions we have the following:

\begin{center}
	\begin{tabular}{|c|c|c|c|c|c|c|c|c|c|}
		\hline 
		& $f$& $\theta$& $\check E$& $\check J_c$ & $\check Y$ & $\check F_R$& $\check F_C$& $\check G$& $N$\\
		\hline\hline 
		$\check J_c$ & $0$ & $\pi^*\theta$  & $E$ & $-E-2\overline R$ & $0$ & $0$ & $0$ & $\overline R$ & $0$ \\
		\hline 
		$\check Y$ & $\pi^*f_C$ & $8\pi^*f_D$ & $\sum_{i=1}^6E_i$ & $0$ & $-\overline D_c-4\pi^*f_D$ & $\overline D_c$ & $\sum_{i=1}^6\overline f_{D,i}$ & $0$ & $0$ \\
		\hline 
		$\check E$ & $0$ & $0$ & $-\nu^*\sigma^*\rho^*h$ & $\check{\ell}_J$ & $\sum_{i=1}^6\check{\ell}_i$ & $\sum_{i=1}^6\check e_i$ & $\check e_7$ & $\sum_{i=1}^6\check g_i$ & $\sum_{i=1}^6n_i$\\
		\hline 
		$\check F_R$ & $0$  & $8f$ & $6f$ & $0$ & $\xi-4f$ & $-3\xi+12f$ & $0$ & $0$& $\xi-6f$\\ 
		\hline 
		$\check F_C$ & $f$ & $0$ & $f$ & $0$ & $6(\xi-f)$ & $0$ & $-\xi$ & $0$& $0$	\\
		\hline 	 
		$\check G$ & $0$ & $8f$ & $6f$ & $\xi-2f$ & $0$ & $0$ & $0$ & $-2\xi+4f$& $\xi-4f$\\
		\hline 
		$N$& $0$ &  $8f$ & $6f$ & $0$ & $0$ & $\xi-2f$& $0$& $\xi-2f$& $-\xi$\\
		\hline
	\end{tabular}
\end{center}

Now consider the class 
\[\eta=\nu^*\gamma-N.\]
We compute $\eta|_{\check J_c}=0$, $\eta|_{\check Y}=0$, $\eta|_{\check F_C}=0$, $\eta|_{\check F_R}=0$, $\eta|_{\check G}=0$, $\eta|_N=\xi+(\gamma\cdot \widehat R)f=\xi-2f$ and $\eta|_{\check E}=2\check{\ell}_J+\sum_{i=1}^6\check g_i+16\nu^*\sigma^*(\pi^*h-e_7)$. These are all nef, thus $\eta$ is nef. Furthermore $\eta$ is zero on $\check J_c$, $\check Y$, $\check F_R$, $\check F_C$, $\check G$, and not big on $N$. These imply that $\eta$ is the positive nef part of a birational Fujita--Zariski decomposition of $\alpha$, and also that $P_{\sigma}(\rho^*\alpha)=\beta$ and $P_{\sigma}(\sigma^*\rho^*\eta)=\gamma$. Remarkably, this process works for all $L_{s,t}$.

\begin{theorem}\label{thm:JxCfinalform}
	With notation as above, consider $0<s<1$ and $0\leq t\leq 1+\frac 12s$ so that $L_{s,t}=s\pi^*q^*\theta+(1-s)\pi^*f-tE$ is pseudoeffective. Then 
\begin{enumerate}
	\item $L_{s,t}$ admits a birational Fujita--Zariski decomposition on $\check X$ with nef positive part
	\begin{align*}
		P_{s,t}&\deq j_{s,t}\check J_c+y_{s,t}\check Y+e_{s,t}\check E+r_{s,t}\check F_R+c_{s,t}\check F_c+g_{s,t}\check G+n_{s,t}N\\
		& = \nu^*\biggl(\sigma^*\biggl(\rho^*P_{\sigma}(L_{s,t})-(j+y-r)F_R-(6y-c)F_C\biggr)-\frac12(j+y-r)G\biggr)-\frac 12(j+y-r)N
	\end{align*}
	where if we denote $e=e_{s,t}$, etc., then
	\begin{align*}
		e&\deq 1+\frac 12s-t & 	j&\deq \min\{1-s,\ e\}\\ 	 
		y&\deq \frac 32\min\{\frac 16s,\ e\} & 
		r&\deq\min\{j+y,\ \frac 13y+e\}\\
		c&\deq\min\{6y,\ e\} &
		g&\deq\min\{j+r,\ \frac 12j+\frac 32e\}=\frac 12(j-y+3r)\\
		n&\deq\min\{r+g,\ 3e\}=-y+3r=\min\{3j+2y,3e\}
	\end{align*}
\smallskip

\item $\vol(L_{s,t})=(P_{s,t}^3)=96jy^2-(j+6y-e)^3+9(j+y-r)^2(j+2r-3e)-(6y-c)^2(3e-6y-2c)$.
\smallskip

\item For fixed $0<s<1$, we have 
\[\inob{x}{L_s}=\biggl\{(t,x,z)\ \mid\ 
	0\leq x\leq 1-s,\ 
%	0\leq z\leq \min\{j+c-e,\ \frac 89c-\frac 13y+r-e\}\\
	0\leq z\leq {\rm min}\bigl\{9e-\frac{81}8x,\ \frac 32s-\frac 98x,\ 1+\frac 12s-e-x,\ 1-s+8e-10x\bigr\}\biggr\}\]
For example 
\[\nob{x}{L_{\frac 12}}=\textup{convex hull}\bigl((0,0,0),\ (\frac 12,\frac 12,0),\ (\frac 34, 0,\frac 34),\ (\frac 54,0,0),\  (\frac{29}{42},\frac{10}{21},\frac 34),\ (\frac 76,0,\frac 34),\ (\frac 23,\frac 12,\frac 16)\bigr)\]	
When we glue all $\nob{x}{L_s}$ as $s$ ranges in $(0,1)$, we obtain the 4D polytope in $(s,t,x,z)$ coordinate space 
\[\textup{convex hull}\bigl((0,0,0,0),\ (1,0,0,0),\ (0,1,0,0),\ (1,\frac 32,0,0),\ (1,\frac 43,0,\frac 43),\ (\frac 37,\frac 4{7},\frac 47,0),\ 
	(\frac 67,\frac 97,0,\frac 97)\bigr)\]
\end{enumerate}	
	\end{theorem}

Figures \ref{fig:CxJ37}, \ref{fig:CxJ12} and \ref{fig:CxJ67} show some examples.

\begin{figure}
	\begin{minipage}{.33\textwidth}
		\centering
		\includegraphics[width=.9\linewidth]{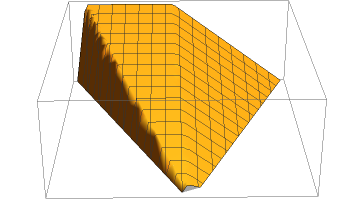}
		\captionof{figure}{$s=\frac 37$}
		\label{fig:CxJ37}
	\end{minipage}%
	\begin{minipage}{.33\textwidth}
	\centering
	\includegraphics[width=.9\linewidth]{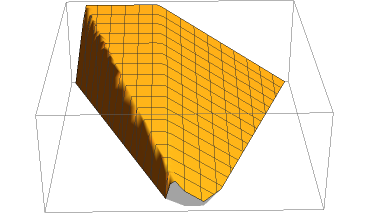}
	\captionof{figure}{$s=\frac 12$}
	\label{fig:CxJ12}
\end{minipage}%
	\begin{minipage}{.33\textwidth}
		\centering
		\includegraphics[width=.9\linewidth]{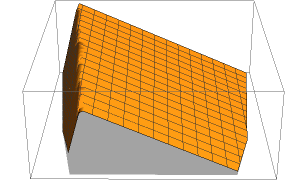}
		\captionof{figure}{$s=\frac 67$}
		\label{fig:CxJ67}
	\end{minipage}
\end{figure}

Before the proof, we first explain how the coefficients were predicted in (1). The pushforward of $P_{s,t}$ to $\overline X$ is $j\overline J_c+y\overline Y+eE$, which equals $P_{\sigma}(L_{s,t})$ by Proposition \ref{prop:firstCxJ}.(5). We have $\rho^*P_{\sigma}(L_{s,t})=j\widetilde J_c+y\widetilde Y+e\widetilde E+(j+y)F_R+6yF_C$. The smallest nonnegative $u$ such that $(\rho^*P_{\sigma}(L_{s,t})-uF_R)|_{F_R}=u\xi+(P_{\sigma}(L_{s,t})\cdot\overline R_c)\cdotp f=u\xi-(6j+4y-6e)f$ is pseudoeffective is $\frac 16(6j+4y-6e)_+=(j+\frac 23y-e)_+=j+y-r$. We deduce that $\sigma_{F_R}(\rho^*P_{\sigma}(L_{s,t}))\geq j+y-r$. We guess that  equality holds. The coefficient of $F_R$ in $P_{\sigma}(\rho^*P_{\sigma}(L_{s,t}))$ would then indeed be be $r$. 

The other coefficients were guessed similarly: For every blow-up $\theta:S'\to S$ with exceptional divisor $E$, we take our candidate positive part $P$ on $S$, then consider on $S'$ the class $\theta^*P-uE$ with smallest possible nonnegative $u$ such that $(\theta^*P-uE)|_E$ is pseudoeffective. 
We can compute $u$ since we understand restrictions to $F_C$, $G$, $N$, and have explicit descriptions of the effective cones of these ruled surfaces.

\begin{proof} (1) 
	The equalities $g=\frac12(j-y+3r)$ and $n=-y+3r$ are elementary verifications. They are consequences of the following identity: If $C=\min\{A,B\}$, then for all $0< t\leq 1$ we have $C=\min\{tA+(1-t)C,B\}$. They imply $y-3r+n=0$ and $j-2g+n=0$.
	
The equality of the two claimed expressions of $P_{s,t}$ follows from $n-r-g=-\frac 12(j+y-r)$ and $g-r-j=-\frac 12(j+y-r)$, and from the relations between strict transform and pullback of divisors for each blow-up.	

	We look to apply the strategy in \ref{subsec:neftest}.
We have $P_{s,t}|_{\check F_R}=(y-3r+n)(\xi-4f)+(6e-2n)f=(6e-2n)f$. 
Similarly, $P_{s,t}|_{\check G}=(j-2g+n)(\xi-2f)+(6e-2n)f=(6e-2n)f$. 
Both restrictions are nef ($n\leq 3e$) and not big.

We also have $P_{s,t}|_{\check F_C}=(6y-c)(\xi-f)+(e-c)f$ and $P_{s,t}|_{N}=(r+g-n)(\xi-2f)+(6e-2n)f$. Since $c=\min\{6y,e\}$ and $n=\min\{r+g,3e\}$, the coefficients are nonnegative, and at least one of them is 0 in each restriction. Since neither $\xi-f$ nor $f$ are big on $F_C$ and neither $\xi-2f$ nor $f$ are big on $N$, we deduce that these restrictions are also nef, but not big.

We have $P_{s,t}|_{\check J_c}=(g-2j)\overline R+(e-j)E$. By Proposition \ref{prop:genus2surface}, nefness is tested by intersecting with $\overline J$ and $E$. The necessary conditions are $-4(g-2j)+6(e-j)\geq0$ and $6(g-2j)-(e-j)\geq0$. The first is equivalent to $g\leq\frac 12j+\frac 32e$ (clear). The second is equivalent to $6g-11j-e\geq 0$. By the definition of $g$, it boils down to checking $e\geq j$ (clear), and $6r\geq 5j+e$. By the definition of $r$, the latter boils down to $j+6y\geq e$ (which follows from $1-s+\frac 32s=1+\frac 12s\geq e$), and to $2y+5e\geq 5j$ (clear from $e\geq j$). We conclude that this restriction is also nef. 

The restriction to $\check Y$ is $-y(\overline D_c+4\pi^*f_D)+e\sum_{i=1}^6E_i+r\overline D_c+c\sum_{i=1}^6\overline f_{D,i}=(r-y)\overline D_c+(6c-4y)\pi^*f_D+(e-c)\sum_{i=1}^6E_i$. The coefficients on the right are nonnegative. The nefness of this restriction is tested by intersecting with the negative curves $\overline D_c$ and $E_i$.
We want to check $-6(r-y)+(6c-4y)+6(e-c)\geq0$ and $r-y-(e-c)\geq 0$. The first is $r\leq\frac 13y+e$ (clear). The second is $(r-y)+c\geq e$, equivalently $j+c\geq e$ and $e-\frac 23y+c\geq e$ both hold. The latter is $c\geq\frac 23y$ which is clear from the definition of $c$ and $e\geq\frac 23y$. The former is equivalent to $j+6y\geq e$ (seen before), and $j+e\geq e$ (clear) both holding. We conclude that $P_{s,t}|_{\check Y}$ is nef.

Finally, the restriction to $\check E$ is $j\check{\ell}_J+y\sum_{i=1}^6\check{\ell}_i+r\sum_{i=1}^6\check e_i+c\check e_7+g\sum_{i=1}^6\check g_i+n\sum_{i=1}^6n_i-e\nu^*\sigma^*\rho^*h$. This is nef by Lemma \ref{lem:P2blow7} below, since the curves that test nefness in $\check E$ also live in $\check J_c$, $\check Y$, $\check F_R$, $\check F_C$, $\check G$, or $N$ where we already checked nefness.

Then $P_{s,t}$ is nef. Since it is not big on any of $\check F_R$, $\check F_C$, $\check G$, or $N$, we must have $P_{\sigma}(P_{s,t}+u\check F_R+v\check F_C+w\check G+zN)=P_{s,t}$ for all $u,v,w,z\geq 0$. By the second expression for $P_{s,t}$, we have that $\nu^*\sigma^*\rho^*P_{\sigma}(L_{s,t})-P_{s,t}$ is a nonnegative combination of $\check F_R$, $\check F_C$, $\check G$, and $N$. It follows that
\[P_{s,t}=P_{\sigma}(\nu^*\sigma^*\rho^*P_{\sigma}(L_{s,t}))=P_{\sigma}(\nu^*\sigma^*\rho^*L_{s,t}).\]

(2) We have $\vol(L_{s,t})=\vol(\nu^*\sigma^*\rho^*L_{s,t})=\vol(P_{\sigma}(\nu^*\sigma^*\rho^*L_{s,t}))=\vol(P_{s,t})=(P_{s,t}^3)$. 
Denote 
\[u\deq j+y-r\quad\text{ and }\quad v\deq 6y-c.\]
Denote also by $\widehat P$, $\widetilde P$, $\overline P=P_{\sigma}(L_{s,t})$ the pushforward of $P_{s,t}$ in each of the lower models $\widehat X$, $\widetilde X$, and $\overline X$ of $\check X$. 

We have $(P_{s,t}^3)=(\nu^*\widehat P-\frac 12uN)^3=(\widehat P^3)-\frac 32u(\nu^*\widehat P^2\cdot N)+\frac 34u^2(\nu^*\widehat P\cdot N^2)-\frac 18u^3(N^3)$. Here $(\nu^*\widehat P^2\cdot N)=(\widehat P\cdot\widehat R)f^2=0$ in $N$, and $(\nu^*\widehat P\cdot N^2)=(\widehat P\cdot\widehat R)f\cdot(-\xi)=-(\widehat P\cdot\widehat R)$, while $(N^3)=((-\xi)^2)=4$ computed in $N$. We deduce
\[(P_{s,t}^3)=(\widehat P^3)-\frac 34u^2(\widehat P\cdot\widehat R)-\frac 18u^3(N^3).\]
Similarly $(\widehat P^3)=(\widetilde P^3)-\frac 34u^2(\widetilde P\cdot\widetilde R)-\frac 18u^3(G^3)$ and $(\widetilde P^3)=(\overline P^3)-3u^2(\overline P\cdot\overline R_c)-u^3(F_R^3)-3v^2(\overline P\cdot\overline C_o)-v^3(F_C^3)$. Then
\[(P_{s,t}^3)=(\overline P^3)-\frac 34u^2\bigl((\widehat P\cdot\widehat R)+(\widetilde P\cdot\widetilde R)+4(\overline P\cdot\overline R_c)\bigr)-u^3(\frac 48+\frac 68+10)-3v^2(\overline P\cdot\overline C_0)-2v^3.\]
Now $\overline P=\pi^*(jJ_c+yY)-(j+6y-e)E$, leading to 
\[(\overline P^3)=(jJ_c+yY)^3-(j+6y-e)^3=(jf+4yq^*\theta)^3-(j+6y-e)^3=96jy^2-(j+6y-e)^3.\]
Furthermore 
\[(\overline P\cdot\overline R_c)=((\pi^*(jf+4yq^*\theta)-(j+6y-e)E)\cdot\overline R_c)=32y-6(j+6y-e)=-(6j+4y-6e).\]
Then $(\widetilde P\cdot\widetilde R)=(\rho^*\overline P-uF_R-vF_C)\cdot\widetilde R=(\overline P\cdot\overline R_c)-((uF_R+vF_C)\cdot \widetilde J_c\cdot F_R)=-(6j+4y-6e)-u(F_R|_{\widetilde J_c}^2)=-(6j+4y-6e)+4u$. And $(\widehat P\cdot\widehat R)=(\sigma^*\widetilde P-\frac 12uG)\cdot\widehat R=(\widetilde P\cdot\widetilde R)-\frac 12u(G\cdot \widehat F_R\cdot G)=(\widetilde P\cdot\widetilde R)-\frac 12u(G|_{\widehat F_R}^2)=-(6j+4y-6e)+4u+u$. Finally, $(\overline P\cdot\overline C_o)=(\pi^*(jJ_c+yY)-(j+6y-e)E)\cdot\overline C_o=e-6y$. In conclusion,
$(P_{s,t}^3)=96jy^2-(j+6y-e)^3-\frac 34(j+y-r)^2(-6(6j+4y-6e)+9(j+y-r))-\frac{45}4(j+y-r)^3-3(6y-c)^2(e-6y)-2(6y-c)^3$, which equals the claimed volume formula.
\smallskip

(3) The plan is to apply Proposition \ref{prop:iNOthreefold}. 

\underline{Claim 1}: If $0<t<1+\frac 12s$, then $P_{s,t}|_{\check E}$ is big. Together with part (1), this verifies the hypotheses of Proposition \ref{prop:iNOthreefold}.

If the nef $P_{s,t}$ is not big on $\check E$, since it was also not big on $\check F_R$, $\check F_C$, $\check G$, $N$, then as in (1) we obtain that $P_{s,t}$ is the positive part of any divisor of form $P_{s,t}+F$ where $F$ is effective supported on $\check E\cup\check F_R\cup\check F_C\cup\check G\cup N$. By the construction of $P_{s,t}$, such a divisor is $\nu^*\sigma^*\rho^*P_{\sigma}(L_{s,t})$ (even without using $\check E$). When we also use $\check E$, the nef divisor $\nu^*\sigma^*\rho^*\pi^*(jf+yq^*4\theta)$ is also of this form. By the uniqueness of the $\sigma$-Zariski decomposition, we deduce $P_{s,t}=\nu^*\sigma^*\rho^*\pi^*(jf+yq^*4\theta)$. Furthermore, $P_{\sigma}(L_{s,t})$ equals $\pi^*(jf+4yq^*\theta)=j\overline J+y\overline Y+(j+6y)E$, forcing $j+6y=e=1+\frac 12s-t$. This is equivalent to $\min\{1-s,1+\frac 12s-t\}+\min\{\frac 32s,9(1+\frac 12s-t)\}=1+\frac 12s-t$, which is impossible in the given range for $s,t$.
\smallskip

Next we determine the bounds in the conclusion of Proposition \ref{prop:iNOthreefold}.

\underline{Claim 2}: In the same range for $s$ and $t$, we have
\[\sup\{u\ \mid\ (P_{s,t}+u\check E)|_{\check E}\text{ is big}\}=\min\{j+c-e,\ \frac 89c-\frac 13y+r-e\}.\]

 By Lemma \ref{lem:P2blow7}.(5), the supremum is the largest $u$ value where we have nonnegativity of the intersection numbers with the 8 generating symmetric nef classes in Lemma \ref{lem:P2blow7}.(4). The conditions where $\epsilon=0$ are consequences of the corresponding inequality for $\epsilon=1$, since $\alpha\cdot(-\check E_7)=c-6y\leq 0$. We are left with
\[\begin{cases}
	j+c-e\geq u\\
	11c+6g-12e\geq 12u\\
	6r+5c-6e\geq 6u\\
	16c+6n-18e\geq 18u 
\end{cases}\] 
One verifies that the fourth condition implies the third, while the first and fourth conditions imply the second. The conclusion follows.
%We see that $18r+15c\geq 16c+6n=16c-6y+18r$. This implies that the third condition is redundant, implied by the fourth. The second condition is $11c+3j-3y+9r-12e\geq 12u$. This is the average of the first condition multiplied by 6, and the fourth condition, hence it is also redundant. We deduce that the pseudoeffectivity condition for $\alpha$ is $u\leq\min\{j+c-e,\frac 89c-\frac 13y+r-e\}$.
\smallskip

\underline{Claim 3}: For $0<t<1+\frac 12s$ and $0\leq u\leq\min\{j+c-e,\ \frac 89c-\frac 13y+r-e\}$, we have
\[(P_{\sigma}((P_{s,t}+u\check E)|_{\check E})\cdot(-\check E|_{\check E}))=j+6y-e-u-\frac{11}8a-9b\ ,\]
where
\[a=a(s,t,u)\deq\frac 1{11}(8j+3y-9r+e+u)_+\quad\text{ and }\quad b=b(s,t,u)\deq (y-r-c+e+u)_+\ .\]

Recall that $\alpha\deq (P_{s,t}+u\check E)|_{\check E}=\check{\ell}_J+y\sum_{i=1}^6\check{\ell}_i+r\sum_{i=1}^6\check e_i+c\check e_7+g\sum_{i=1}^6\check g_i+n\sum_{i=1}^6n_i-(e+u)\nu^*\sigma^*\rho^*h$ on $\check E$. With the notation of Lemma \ref{lem:P2blow7}, this is
\[\alpha=j\check L+y\check L_7+r\check E+c\check E_7+g\check G+n N-(e+u)H\]
Since $P_{s,t}|_{\check E}$ is nef, the class $\alpha$ can only be negative on $\check L$ and $\check L_i$.
We compute
\begin{align*}
	(\alpha\cdot\check{L})&=-11j+6g-e-u=-8j-3y+9r-e-u\\
	(\alpha\cdot\check{L}_i)&=-y+r+c-e-u
\end{align*}

Then $\alpha-a\check{\ell}_j-b\sum_{i=1}^6\check{\ell}_i$ necessarily dominates the positive part of the Zariski decomposition of $\alpha$. Its intersections with $\check{L}$ and the $\check{L}_i$ are $-8j-3y+9r-e-u+11a$ and $-y+r+c-e-u+b$ respectively. They are nonnegative. The intersections with $\check E_i$ and $\check G_i$ are $y-b-3r+n=-b$ and respectively $j-a-2g+n=-a$.

If $a$ and $b$ are both 0, then $\alpha$ is nef and $P_{\sigma}(\alpha)=\alpha$ and $(P_{\sigma}(\alpha)\cdot H)=j+6y-e-u$.

If $a$ and $b$ are both positive, we deduce that the support of $N_{\sigma}(\alpha)$ contains $\check L\cup\check L_7\cup\check E\cup\check G$. This is a union of 19 curves with linearly independent classes in the 20 dimensional $N^1(\check E)$. Up to scaling, there is only one nef class that is zero on all of these. This is the symmetric class \[\eta\deq 2\check\ell+16(\nu^*\sigma^*(\rho^*h-e_7))+\sum_{i=1}^6\check g_i\] by the intersection table after Lemma \ref{lem:P2blow7}. Then $P_{\sigma}(\alpha)=z\eta$ for some $z\geq 0$ and $N_{\sigma}(\alpha)$ is of form $M\deq x\check L+x_7\check L_7+x_e\check E+x_g\check G$.
We obtain a linear system of equations for the coefficients. 
\[\begin{cases}
	-11x+6x_g=(M\cdot\check L)=(\alpha\cdot\check L)=-a\\
	x-2x_g=(M\cdot\check G_i)=(\alpha\cdot\check G_i)=0\\
	-x_7+x_e=(M\cdot\check L_i)=(\alpha\cdot\check L_i)=-b\\
	x_7-3x_e=(M\cdot \check E_i)=(\alpha\cdot\check E_i)=0
\end{cases}\]
Note that $(\alpha\cdot\check L)=-a$ and $(\alpha\cdot\check L_i)=-b$ hold only when $(\alpha\cdot\check L)$ and $(\alpha\cdot\check L_i)$ are nonpositive.
Solving the system gives $M=\frac{11}8a\check L+\frac 32b\check L_7+\frac 12b\check E+\frac{11}{16}a\check G$. 
Then $32z=z\cdotp(\eta^2)=(\alpha\cdot\eta)=16c+6n-18(e+u)\geq 0$. We used that $\alpha$ is pseudoeffective and $\eta$ is nef and orthogonal to the components of $M$. Then $(P_{\sigma}(\alpha)\cdot H)=(\alpha\cdot H)-\frac {11}8a-6\cdot\frac 32b=(\alpha\cdot H)-\frac{11}8a-9b$.

If $a=0$ and $b>0$, then as in the previous case the negative part contains $\check L_7+\check E$. We claim that the negative part of the Zariski decomposition is $M\deq\frac 32b\check L_7+\frac 12b\check E$. 
Let $\eta\deq\alpha-M$. Then its intersection with:
\begin{itemize}
	\item $\check L$ is $-11j+6g-e-u\geq 0$ by the assumption that $a=0$.
	\item $\check L_i$ is $(-y+r+c-e-u)-(-\frac 32b+\frac 12b)=-b+\frac 32b-\frac 12b=0$ by the definition of $b$ and the assumption $b>0$.
	\item $\check E_i$ is $(y-3r+n)-(\frac 32b-3\frac 12b)=0$.
	\item $\check G_i$ is $(j-2g+n)-0=0$.
	\item $\check E_7$ is $(-c+6y)-6\cdot\frac 32b=-c+6y-9b=-c+6y-9(y-r-c+e+u)=8c-3y+9r-9e-9u$. This is nonnegative by (4).
	\item $N_i$ is $(g+r-n)-\frac 12b=g+r-n-\frac 12(y-r-c+e+u)=\frac 12(j+c-e-u)\geq 0$ by (4). 
\end{itemize}
The claim follows. We also have $(P_{\sigma}(\alpha)\cdot H)=(\alpha\cdot H)-6\cdot\frac 32b=(\alpha\cdot H)-9b$.

When $a>0$ and $b=0$, we similarly prove that $N_{\sigma}(\alpha)=\frac{11}8a\check L+\frac{11}{16}a\check G$.
The nontrivial check is that $(\alpha-\frac{11}8a\check L-\frac{11}{16}a\check G)\cdot N_i=(g+r-n)-\frac{11}{16}a\geq0$. It boils down to $5y+r\geq e+u$ which follows from $5y\geq\frac 56c$ and the third inequality in the proof of Claim 2. Then $(P_{\sigma}(\alpha)\cdot H)=(\alpha\cdot H)-\frac{11}8a$.
\smallskip

From Proposition \ref{prop:iNOthreefold} we deduce that for fixed $0<s<1$ the body $\inob{x}{L_s}$ is determined in $(t,x,z)$ coordinates by the inequalities
\[\begin{cases}
	0\leq t\leq 1+\frac 12s\\
	0\leq x\leq \min\{j+c-e,\ \frac 89c-\frac 13y+r-e\}\\
	0\leq z\leq j+6y-e-x-\frac{11}8a-9b
\end{cases}\]
The passage from these inequalities to those in the conclusion of (3) starts off as a brute force argument where the piece-wise linear functions $j,y$, $\frac 13 n= r-\frac 13y=\min\{j+\frac 23y,e\}$, $c$, $-a,-b$ are minima between linear functions. After careful manipulation they never appear with a negative sign in our formulas, and then each branch determines a half-space cutting out $\inob{x}{L_s}$. For example one of the branches of the upper bound on $z$ is $j+6y-e-x-\frac{1}8(8j-3n+e+x)-9(y-r-c+e+x)=\frac{21}8n+9c-\frac{81}{10}(e+x)$. The branch functions $n$ and $c$ appear with nonnegative coefficient.

We eliminate $a,b,c,n,y,j$ one by one in this order by replacing them with their branches, and then remove redundant inequalities. This leaves us with a small number of inequalities that we plug into \cite{Polymake} to produce the minimal set of defining inequalities in the conclusion of (3). Polymake also computed the vertices of the 4D body. Furthermore it computes that its volume is $\frac 1{12}$ which agrees with $\int_0^1\vol_{\bb R^3}\inob{x}{L_s}ds=\int_0^1\frac 16(L_s^3)ds=\int_0^1s^2(1-s)ds$ as expected.
\end{proof}

\begin{lemma}\label{lem:P2blow7}
	Let $p_1,\ldots, p_6$ be points on a line $L$ in $X=\bb P^2$. Let $p_7\in\bb P^2$ not on $L$. For $1\leq i\leq 6$, let $L_i$ be the line joining $p_7$ and $p_i$. Consider the following sequence of blow-ups:
	\begin{itemize}
		\item $\rho:\widetilde{X}\to\bb P^2$ is the blow-up of the 7 points $p_i$ with exceptional divisors $E_i$.
		\smallskip
		
		\item $\sigma:\widehat X\to\widetilde X$ is the further blow-up of the 6 points $E_i\cap\widetilde L$ with exceptional divisors $G_i$.
\smallskip

		\item $\nu:\check X\to\widehat X$ is the blow-up of the 6 points $G_i\cap\widehat E_i$ with exceptional divisors $N_i$. 
	\end{itemize} 
\begin{enumerate}
	\item There is a natural action of $\frak S_6$ on $N^1(\check X)$ that respects the intersection pairing. The invariant subspace has a basis given by $\check L$, $\check E\deq\sum_{i=1}^6\check E_i$, $\check E_7$, $\check G\deq\sum_{i=1}^6\check G_i$, and $N\deq\sum_{i=1}^6N_i$. For a curve $C$ in $\bb P^2$, in $\widetilde X$, or $\widehat X$, we denoted by $\check C$ the strict transform in $\check X$.
\smallskip

	\item A $\frak S_6$-symmetric class $\alpha$ is nef iff the intersection numbers 
	\[\alpha\cdot\check L,\quad \alpha\cdot\check L_1,\quad \alpha\cdot\check E_1,\quad \alpha\cdot\check E_7,\quad \alpha\cdot\check G_1,\quad\alpha\cdot N_1\] are nonnegative.
\smallskip

	\item The $\frak S_6$-symmetric slice of the pseudoeffective cone of $\check X$ is generated by 
	\[\check L,\quad\check L_7\deq\sum\nolimits_{i=1}^6\check L_i,\quad \check E,\quad\check E_7,\quad\check G,\quad N.\]
	In particular every symmetric pseudoeffective class is effective. 
\smallskip

	\item Denote by $H$ the pullback of the line class. The generators of the $\frak S_6$-symmetric slice of $\Nef(\check X)$ are the 8 classes 
	\[\begin{cases}H-\epsilon\check E_7\\
		\check L+11(H-\epsilon\check E_7)\\
		\check L+5(H-\epsilon\check E_7)+\check G+N\\
		2\check L+16(H-\epsilon\check E_7)+\check G
	\end{cases}\]	
	where $\epsilon\in\{0,1\}$.
\smallskip

	\item A $\frak S_6$-symmetric class $\alpha$ is pseudoeffective iff its intersection numbers with the 8 generators of the symmetric slice of the nef cone are nonnegative.
\end{enumerate}
\end{lemma}

Figure \ref{fig:P27} shows the configuration of strict transform lines in $\widehat X$ (only $i=1,6$ are shown of the six $1\leq i\leq 6$) and their self-intersections.

\begin{figure}
	\centering
	\includegraphics[scale=0.1]{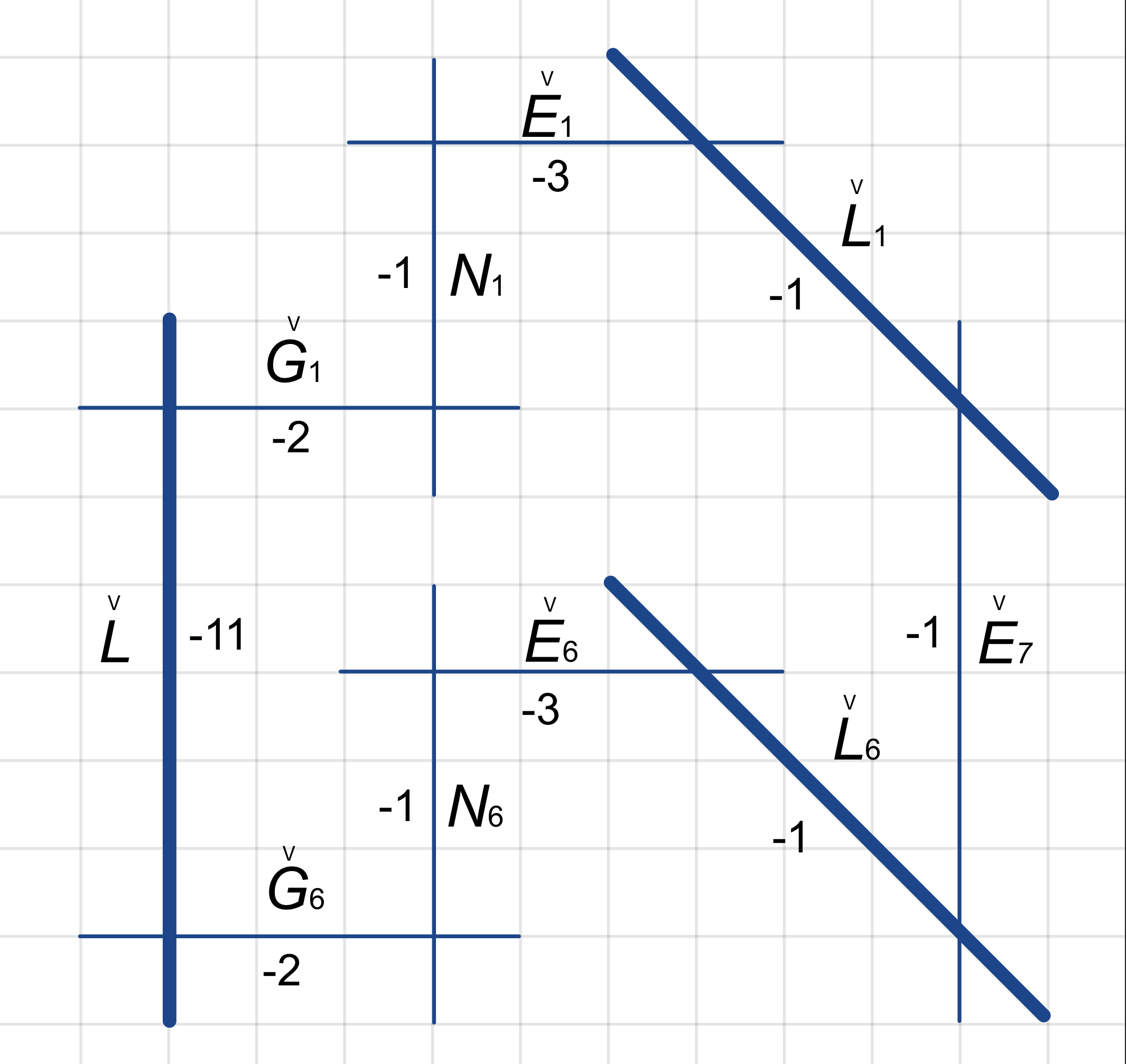}
	\caption{Lines in $\check X$}
	\label{fig:P27}
\end{figure}

\begin{proof} 
(1) $N^1(\check X)$ is $20$-dimensional, generated by $\check L$, and by the $19=1+3\cdot 6$ (strict transforms of) exceptional divisors $\check E_i$ (including $\check E_7$), $\check G_i$, $N_i$. The action permutes the $\check E_i$ (for $1\leq i\leq 6$), $\check G_i$, and respectively $N_i$-coordinates. The other claims of (1) are straightforward. 

Note that the action of $\frak S_6$ on $N^1(\check X)$ is natural, but does not come from obvious automorphisms of $\check X$. In particular it is not clear that the full $20$-dimensional positive cones of $\check X$ are preserved by the action. 

(2)	Only one implication needs proving. Assume that the intersection numbers are nonnegative. For $\alpha=a\check L+b\check E+c\check E_7+d\check G+eN$, the conditions are
\begin{equation}\label{eq:P27ptscond}
\begin{cases}
-11a+6d\geq 0 & b+c\geq 0\\
-3b+e\geq 0 & c\leq 0\\
a-2d+e\geq 0 & b+d-e\geq 0
	\end{cases}
\end{equation} 
We have
\[H=\nu^*\sigma^*\rho^*L=\check L+\check E+2\check G+3N.\]
We can rewrite
\begin{equation}\label{eq:P27form}\alpha=(a+c)\check L+(-c)(H-\check E_7)+(b+c)\check E+(d+2c)\check G+(e+3c)N.
\end{equation}
Note that $H-\check E_7=\nu^*\sigma^*(\rho^*L-E_7)$ is nef.
We claim that the last five inequalities in \eqref{eq:P27ptscond} imply that the coefficients in \eqref{eq:P27form} are all nonnegative. If true, then $\alpha$ is effective. If  $\alpha\cdot\check L$, $\alpha\cdot\check E_1$, $\alpha\cdot\check G_1$, $\alpha\cdot N_1$ are nonnegative, then $\alpha$ is nef.

For the claim, we have $-c\geq 0$ and $b+c\geq 0$ as part of the conditions. Furthermore $e+3c\geq3b+3c\geq 0$, $d+2c\geq e-b+2c\geq 2b+2c\geq 0$, and finally $a+c\geq 2d-e+c=2(d-e)+e+c\geq -2b+e+c\geq b+c\geq 0$.

(3) Let $V$ be the 5-dimensional symmetric subspace of $N^1(\check X)$ from (1). The intersection pairing on $V$ is nondegenerate of signature $(1,4)$.
In (2) we may replace $\check L_1$, $\check E_1$, $\check G_1$, $N_1$ with their symmetrized versions $\check L_7$, $\check E$, $\check G$, $N$ (which are still effective) without changing the conclusion. It follows that when working in $V$, the dual of $\Nef(\check X)\cap V$ is generated by $\check L$, $\check L_7$, $\check E$, $\check E_7$, $G$, $N$, hence it is contained in $\Eff(\check X)\cap V$. Every psuedoeffective class in $V$ is an element of the dual of the cone $\Nef(\check X)\cap V$. The claim follows.

(4). We are looking for the vertices of the cone cut out by the inequalities in (2). This computation can be carried out in \cite{Polymake}. (5) is analogous to (3).
\end{proof}
% (4). Every nonzero nef class $\alpha=a\check L+b\check E+c\check E_7+d\check G+eN$ must have $a>0$, otherwise $\alpha\cdot H=0$ and then $\alpha^2<0$ by Hodge index. Up to scaling we may assume $a=1$, and we want to determine the vertices of the polytope (actually polygon) determined by \eqref{eq:P27ptscond} in the affine hyperplane $a=1$. 

%If we set aside the first condition $d\geq\frac{11}6$, the remaining 5 conditions determine a simplex in the 4-dimensional hyperplane. Its vertices are found by solving the 5 linear systems of 4 equations  with 4 unknowns $b,c,d,e$ obtained by removing conditions $1$ and $i$ from \eqref{eq:P27ptscond} for $2\leq i\leq 6$, and replacing the inequalities by equalities. We find the 5 solutions $(1,0,0,0,0)$, $(1,0,0,1,1)$, $(1,0,0,\frac 12, 0)$ and $(1,1,0,2,3)$, $(1,1,-1,2,3)$ for $(a,b,c,d,e)$.

% The polygon that we want is the slice $d\geq\frac{11}6$ of this simplex. Of the 5 simplex vertices, only the last two satisfy $d\geq\frac{11}6$. They correspond to $H=\check L+\check E+2\check G+3N$ and $H-\check E_7$ which are indeed nef. The other vertices of the polygon are obtained by using slicing $d=\frac{11}6$ to slice every edge joining one the first 3 vertices and one of the last 2. This gives the other 6 claimed generators (up to scaling). 

%(5) can be proved analogously to (3). 

We also note the intersection numbers:
\begin{center} \begin{tabular}{|l|c|c|c|c|c|c|c|}
	\hline 
	& $\check L$& $\check L_7$& $\check E$& $\check E_7$& $\check G$& $N$& $H$\\
	\hline\hline
	$H$& 1&6&0&0&0&0&1\\
	\hline 
	$H-\check E_7$& 1&0&0&1&0&0&1\\
	\hline 
	$\check L+11H$& 0&66&0&0&6&0&12\\ 
	\hline 
	$\check L+11(H-\check E_7)$&0&0&0&11&6&0&12\\
	\hline 
	$\check L+5H+\check G+N$& 0&30&6&0&0&0&6\\
	\hline
	$\check L+5(H-\check E_7)+\check G+N$& 0&0&6&5&0&0&6\\
	\hline 
	$2\check L+16H+\check G$& 0&96&0&0&0&6&18\\
	\hline
	$2\check L+16(H-\check E_7)+\check G$&0&0&0&16&0&6&18\\
	\hline 
\end{tabular}
\end{center}

\section{An application}

We illustrate the usefulness of the examples in Theorems \ref{thm:CP}, \ref{thm:CCCfinalform} and \ref{thm:JxCfinalform} as testing ground for open questions. Here we disprove a question of \cite{FLgeneralities} that aims to read an invariant of a curve class from a generic iNObody.

In \cite{Ful21}, the first named author introduced a notion of Seshadri constant for a movable curve class at a point. It is a dual notion to the classical Seshadri constant of nef divisors. If $x\in X$ is a smooth point and $C\in{\rm Mov}_1(X)$ is a \emph{movable curve class}, i.e., a class in the closure of the cone generated by irreducible curves that deform in families that cover $X$, then
\[\epsilon(C;x)\deq\sup\{t\geq 0\ \mid\ \pi^*C-t\ell\in{\rm Mov}_1({\rm Bl_xX})\}=\inf_D\frac{C\cdot D}{\mult_xD}\ ,\]
as $D$ ranges through effective Cartier divisors whose support contains $x$, and where $\ell$ is a line in $E$. The last equality is a consequence of \cite{BDPP}. Curve classes obtained as complete intersection of nef divisor classes are movable (see also Lemma \ref{lem:createmovable} below). We have the following interesting connection between the Seshadri constant of complete (self-)intersection curve classes and convex geometry:

\begin{proposition}[{\cite[Theorem 6.2.(3)]{FLgeneralities}}]
	Let $X$ be a complex projective manifold of dimension $n$, let $L$ be an ample divisor on $X$, and $x\in X$ a very general point. Then 
	\begin{equation}\label{eq:conjecture}\epsilon((L^{n-1});x)\geq (n-1)!\cdot{\rm vol}_{\bb R^{n-1}}{\rm pr}_1(\inob{x}{L})\ ,\end{equation} 
	where here ${\rm pr}_1:\bb R^n\to \bb R^{n-1}$ denotes the projection on the last $n-1$ components, i.e., forgetting $\nu_1$.
\end{proposition}

 Note that \cite[Theorem 6.2]{FLgeneralities} is more general and phrased differently. The formulation above avoids introducing extra terminology.
In \cite[\S6]{FLgeneralities} we observed that \eqref{eq:conjecture} is an equality when $X$ is a surface, or when $\inob{x}{L}$ is a simplex. We also showed that equality fails when $X$ is the Jacobian of a hyperelliptic curve of genus 3 with its natural Theta polarization.
A feature that sets apart the hyperelliptic Jacobian threefold from the other cases mentioned above is that ${\bf B}_+(\pi^*L-tE)$ can have components fully supported in $E$ on ${\rm Bl}_xX$. 
It follows from our work in Theorems \ref{thm:CP}, \ref{thm:CCCfinalform}, and \ref{thm:JxCfinalform} that for the box-product polarizations $L$ considered there we never find purely exceptional components in ${\bf B}_+(L_t)$, essentially because the curves that we blow-up are never fully contained in the strict transform of $E$.
Thus we might be led to believe that \eqref{eq:conjecture} is an equality for the curve classes $(L^2)$ we work with here. We show that this is only sometimes true.

\begin{theorem}\label{thm:Seshdaricurveconjecture}$ $
	\begin{enumerate}[(1)]
		\item With the notation of Theorem \ref{thm:CP},
$\epsilon((L^2);x)=\min\{2ab,b^2\}$.
		The projection of $\inob{x}{L}$ on the last two components is
		\begin{enumerate}[i.]
			\item The standard simplex of size $b$, if $a\geq b$ and \eqref{eq:conjecture} is an equality in this case.
			\item The right trapezoid with vertices at $(0,0)$, $(a,0)$, $(a,b-a)$, $(0,b)$, if $a<b$ and \eqref{eq:conjecture} is always strict in this case.
		\end{enumerate}
\smallskip		
	\item With the notation of Theorem \ref{thm:CCCfinalform}, 
$\epsilon((L^2);x)=2d_2d_3$.
	The projection of $\inob{x}{L}$ is the right trapezoid with vertices $(0,0)$, $(d_3,0)$, $(d_3,d_2-d_3)$, $(0,d_2+d_3)$ and \eqref{eq:conjecture} is always an equality for this class of examples.
\smallskip

\item With the notation of Theorem \ref{thm:JxCfinalform}, for all $s\in(0,1)$ we have
$\epsilon((L_s)^2;x)=\min\bigl\{2s^2,\frac 83s(1-s)\bigr\}$.
Furthermore \eqref{eq:conjecture} is an equality when $s\in(0,\frac 37]$, but it is strict for example when $s=\frac 12$.
	\end{enumerate}
\end{theorem}

For the Seshadri constant computation we use the following method for creating movable curve classes:

\begin{lemma}\label{lem:createmovable}
	Let $X$ be a smooth projective variety of dimension $n$. Let $\eta_1,\ldots,\eta_{n-2}$ be nef divisor classes, and let $\mu$ be a movable divisor class. Then $C\deq\eta_1\cdots\eta_{n-2}\cdot\mu$ is a movable curve class. In particular, if $n=3$, and $L$ is nef on $X$, then $\epsilon((L^2);x)\geq\epsilon(L;x)\cdot\nu(L;x)$.
\begin{proof} 	
	This is a consequence of \cite{BDPP}. If $Z$ is a prime divisor in $X$, then $\mu|_Z$ is a pseudoeffective divisor class since if $D$ moves in a linear series without fixed components, it is clearly linearly equivalent to a divisor $D'\geq 0$ that meets $Z$ properly. Furthermore ${\eta_i}|_Z$ is nef. It follows that $(C\cdot Z)={\eta_1}|_Z\cdots{\eta_{n-2}}|_Z\cdot\mu|_Z\geq 0$. For the last part, observe that $\pi^*(L^2)-\epsilon(L;x)\cdot\nu(L;x)\cdot \ell=(\pi^*L-\epsilon(L;x)E)(\pi^*L-\nu(L;x)E)$.
\end{proof} 
\end{lemma}

\begin{proof}[Proof of {Theorem \ref{thm:Seshdaricurveconjecture}}]
	(1). For the Seshadri constant computation, we want to find the largest $t$ such that $\pi^*(L^2)-t\ell$ is movable on $\overline X$. By \cite{BDPP}, movability for curve classes is tested by intersecting with the generators of the effective cone of divisors. By intersecting $\pi^*(L^2)-t\ell=\pi^*(2ab\cdot\bb P^1_x+b^2\cdot C_x)-t\ell$ with $\overline f$, $\overline H$, and $E$ from Proposition \ref{prop:blCxP2positive}, the conclusion follows. 
	The convex body projections are immediate from Theorem \ref{thm:CP}. The other statements are an easy direct verification.
	\smallskip
	
	(2) If $\pi^*(L^2)-t\ell$ is movable, it has nonnegative intersection with the effective divisors $\overline f_1,\overline f_2,\overline f_3$ and $E$. Since $(L^2)=2d_1d_2f_1f_2+2d_1d_3f_1f_3+2d_2d_3f_2f_3$, we obtain $0\leq t\leq\min\{2d_1d_2,\ 2d_1d_3,\ 2d_2d_3\}=2d_2d_3$. Thus $\epsilon((L^2);x)\leq 2d_2d_3$.
	Since $X=C_1\times C_2\times C_3$ may have Picard rank bigger than $3$, we cannot immediately deduce sufficiency of the conditions above from the duality result of \cite{BDPP} alone.
	
	The classes $f_1f_2$, $f_1f_3$, $f_2f_3$ are movable since they are intersections of nef classes. It is easy to see that as a function on the movable cone of curves, the Seshadri constant of curve classes $\epsilon(\cdot;x)$ is homogeneous and superadditive. For the desired reverse inequality, since $(L^2)-d_2d_3(f_1+f_2+f_3)^2$ is movable, it is sufficient to prove that $\epsilon((f_1+f_2+f_3)^2;x)\geq 2$. For this we may use Lemma \ref{lem:createmovable} for $L'=f_1+f_2+f_3$ given that
	$\epsilon(L';x)=1$ and $\nu(L';x)=2$ from Theorem \ref{thm:CCCfinalform}.
	
	The computation of the convex body projection and its area are straightforward, using Theorem \ref{thm:CCCfinalform}.
\smallskip

(3). 	By abuse we continue to denote by $\theta$ and $f$ the obvious pullbacks to $\overline X$. 
We prove that a curve class $C\in{\rm Span}(\theta^2,\theta f,\ell)\subset N_1(\overline X)_{\bb R}$ is movable if and only if $C$ has nonnegative intersection with the effective divisors $E$, $\overline J_c\equiv f-E$, and $\overline Y\equiv 4\theta-6E$ that generate $\Eff(\overline X)$ by Proposition \ref{prop:firstCxJ}. As in $(2)$, this is not a formal consequence of \cite{BDPP} when $\overline X$ has Picard rank bigger than 3. 

The conditions are clearly necessary by \cite{BDPP}. 
For sufficiency, one computes that the cone in ${\rm Span}(\theta^2,\theta f,\ell)$ determined by the nonnegativity of the 3 intersections is generated by the classes $2\theta^2+3\theta f-4\ell$, $\theta f$, $\theta^2$. It is enough to prove that these classes are all movable. The last two are intersections of nef divisor classes, while the first is proportional to the product $(4f+3\theta-4E)(2f+3\overline Y)=(3\theta+4f-4E)(12\theta+2f-18E)$ between a nef divisor class and a movable divisor class by Proposition \ref{prop:firstCxJ}. Conclude from Lemma \ref{lem:createmovable}.

From the above, the class $\pi^*(L_s^2)-t\ell=s^2\theta^2+2s(1-s)\theta f-t\ell$ is movable if and only if it has nonnegative intersection against the effective divisors $E$, $\overline J_c\equiv f-E$, and $\overline Y\equiv 4\theta-6E$. These are equivalent to $0\leq t\leq\min\{2s^2,\frac 83s(1-s)\}$. 	

From the example $s=\frac 12$ in Theorem \ref{thm:JxCfinalform}.(4), the projection on the coordinate plane $(x,z)$ has vertices $(0,0)$, $(\frac 12,0)$, $(\frac 12,\frac 16)$, $(\frac{10}{21},\frac 3{14})$, $(0,\frac 34)$. Twice its area is $\frac{59}{126}<\epsilon((L_s)^2;s)=\frac 12$. The last equality is by Proposition \ref{prop:firstCxJ}.$(4)$.

We prove that \eqref{eq:conjecture} is an equality when $s\in(0,\frac 37]$.
Projecting the 4D polytope in coordinates $(s,t,x,z)$ described in Theorem \ref{thm:JxCfinalform} on the 3D space where we forget the $t$ coordinate, we obtain the polytope with the five vertices $(0,0,0)$, $(1,0,0)$, $(1,0,\frac 43)$, $(\frac 37,\frac 47,0)$, $(\frac 67,0,\frac 97)$. In this 3D polytope, the slice areas where $s$ is constant are precisely the projections of interest for \eqref{eq:conjecture}.
With \cite{Polymake} we compute that the region $s\leq\frac 37$ of this polytope is cut out by the inequalities $x,z\geq 0$, $\frac 34x+\frac 23z\leq s\leq\frac 37$. Fixing $s$, we obtain that our projections of interest for $s\in(0,\frac 37]$ are triangles with vertices $(0,0)$, $(\frac 43s,0)$, $(0,\frac 32s)$. Twice their areas are $2s^2$. For $s\leq\frac 47$, we also had $\epsilon((L_s^2);x)=2s^2$. The conclusion follows.
\end{proof} 

\begin{remark}
Theorem \ref{thm:Seshdaricurveconjecture} does not use \eqref{eq:conjecture}, in particular it does not use the very general assumption on $x$. We do not know if this hypothesis can be removed from the relevant result \cite[Theorem 6.2.(3)]{FLgeneralities}.	
	
%The astute reader might have observed that when $s\leq\frac 37$, by the proof above and by Proposition \ref{prop:firstCxJ}, the projection of $\inob{x}{L_s}$ is a nonstandard simplex of lengths $\epsilon(L_s;x)=\frac 43s$ and $\nu(L_s;x)=\frac 32s$. Furthermore, in this case the inequality in Lemma \ref{lem:createmovable} is an equality. A more general simplicial criterion for the projection follows from \cite[Theorem 6.2 and Theorem 1.2]{FLgeneralities}. We state it here in dimension 3: 

%Assume that $X$ is a smooth complex projective threefold, $L$ is ample on $X$, and $x\in X$ a very general point. For all $t\in(0,\mu(L;x))$, assume that all components of ${\bf B}_+(\pi^*L-tE)\subset\overline X$ meet $E$, but none are not contained in $E$. This assumption is met in all examples in Theorem \ref{thm:Seshdaricurveconjecture}. Assume furthermore that $\epsilon((L^2);x)=\epsilon(L;x)\cdot\nu(L;x)$ for very general $x\in X$. Then ${\rm pr}_1(\inob{x}{L})$ is the simplex with vertices $(0,0)$, $(\epsilon(L;x),0)$, $(0,\nu(L;x))$.
\end{remark}

\end{document}